\documentclass[11pt,a4paper]{article}
\usepackage{amsmath,amsthm,amssymb,amscd,epsfig}

\usepackage{amsmath ,amsthm,epsfig,amssymb,graphicx}
\usepackage{verbatim}
\usepackage{amssymb}

\begin{document}

\title{ INTEGRAL STABILITY OF CALDER\'ON INVERSE CONDUCTIVITY PROBLEM IN THE PLANE}

\author{Albert Clop \and Daniel Faraco
 \and Alberto Ruiz}
\date{}

{\allowdisplaybreaks \sloppy \theoremstyle{plain}
\newtheorem{Theorem}{Theorem}[section]
\newtheorem{Lemma}[Theorem]{Lemma}
\newtheorem{Cor}[Theorem]{Corollary}
\newtheorem{question}[Theorem]{Question}
\newtheorem{example}[Theorem]{Example}
\newtheorem{Prop}[Theorem]{Proposition}
\theoremstyle{definition}
\newtheorem{Def}[Theorem]{Definition}
\newtheorem{Rem}[Theorem]{Remark}
\newtheorem{Prob}{Problem}
\numberwithin{equation}{section}
\def\halmos{{\ \vbox{\hrule\hbox{\vrule height1.3ex\hskip0.8ex\vrule}\hrule}}\par \medskip}
\def\Xint#1{\mathchoice
{\XXint\displaystyle\textstyle{#1}}%
{\XXint\textstyle\scriptstyle{#1}}%
{\XXint\scriptstyle\scriptscriptstyle{#1}}%
{\XXint\scriptscriptstyle\scriptscriptstyle{#1}}%
\!\int}
\def\XXint#1#2#3{{\setbox0=\hbox{$#1{#2#3}{\int}$}
\vcenter{\hbox{$#2#3$}}\kern-.5\wd0}}
\def\ddashint{\Xint=}
\def\dashint{\Xint-}
\def \t{\tilde}
\def \r{\tilde{R}}
\def \D{\mathbb{D}}
\def \div{\operatorname{div}}
\def \Re{\operatorname{Re}}
\def \Im{\operatorname{Im}}
\def \supp{\operatorname{supp}}
\def \diam{\operatorname{diam}}
\def \C{\mathbb{C}}
\def \N{\mathbb N}
\def\Lip{\operatorname{Lip}}
\def\R{\mathbb{R}}
\def\2L{\Lambda_{\tilde{\gamma}}}
\def\1L{\Lambda_{\gamma}}
\def \c{\overline}
\def \d{\partial_z}
\def \dc{\partial_{\overline z}}
\def \cd{\overline{\partial_z}}
\def \dk{\partial_{\overline k}}
\def \kd{\overline{\partial_k}}
\def \dx{\partial_x}
\def \H{H^{1/2}}
\maketitle

\begin{abstract}
It is proved that, in two dimensions, the Calder\'on inverse
conductivity problem in Lipschitz domains is stable in the $L^p$
sense when the conductivities are uniformly bounded in any
fractional Sobolev space $W^{\alpha,p}$ $\alpha>0, 1<p<\infty$.
\end{abstract}

\section{Introduction}

 \noindent Calder\'on inverse problem consists in the
determination of an isotropic $L^{\infty}$ conductivity
coefficient $\gamma$ on $\Omega$ from boundary measurements. These
measurements are given by the Dirichlet to Neumann map
$\Lambda_\gamma$, defined
 for a function $f$ on $\partial \Omega$ as the Neumann value
$$\Lambda_\gamma(f) = \gamma\frac{\partial}{\partial\nu}u,$$
where $u$ is the solution of the Dirichlet boundary value problem
\begin{equation}\label{cond}
\begin{cases}
\nabla\cdot(\gamma\nabla u)=0\\
u_{\vline \partial \Omega}=f
\end{cases}
\end{equation}
and $\frac{\partial}{\partial \nu}$ denotes the outer normal derivative. For
 general domain and conductivities where the pointwise definition
 $\gamma\frac{\partial}{\partial \nu}u$ has no meaning, the Dirichlet to Neumann map \begin{equation}
\Lambda_\gamma :\H(\partial \Omega) \to H^{-1/2}(\partial \Omega)
\end{equation}
can be defined by
\begin{equation}\label{weakformulation}
\langle \Lambda_\gamma(f), \varphi_0 \rangle= \int_\Omega
\gamma \nabla u\cdot \nabla \varphi
\end{equation}
where $\varphi \in W^{1,2}(\Omega)$ is a function such that $\varphi_{\vline \partial \Omega}=\varphi_0$ in
the sense of traces.\\
\\
Since the foundational work of Calder\'on, the research of the question has been
 very intense but it is not until 2006 when, by means of quasiconformal mappings,
 K. Astala and L. P\"aiv\"arinta in \cite{AP} were able to establish the injectivity of the map
$$\gamma\to \Lambda_\gamma$$
for an arbitrary $L^\infty$ function bounded away from zero. Previous
 planar results were obtained in \cite{Nach} and \cite{SyU2}. In higher dimensions,
 the known results on uniqueness require some extra a priori regularity on $\gamma$
 (basically some control on $\frac{3}{2}$ derivatives of $\gamma$,
 see \cite{SyU}, \cite{B}, \cite{PPU} and \cite{BT}.)\\
\\
A relevant question (specially in applications) is the
 stability of the inverse problem, that is, the continuity of the inverse map
$$\Lambda_\gamma\to \gamma.$$
For dimension $n>2$, the known results are due  to Alessandrini \cite{A}, \cite{A1}. There the author proved stability under the extra assumption $\gamma \in W^{2,\infty}$. In the  planar case, $n=2$, the situation is different. Liu proved stability for conductivities in $W^{2,p}$ with $p>1$ in \cite{Liu}. In \cite{BBR}, stability was obtained when $\gamma \in {\mathcal C}^{1+\alpha}$ with $\alpha>0$. Recently, Barcel\'o, Faraco and Ruiz \cite{BFR} obtained  stability under the weaker assumption $\gamma\in{\cal C}^\alpha$, $0<\alpha<1$. Precisely, they prove that for any two conductivities $\gamma_1,\gamma_2$ on a Lipschitz domain $\Omega$, with a priori bounds $\frac{1}{K}\leq\gamma_i\leq K, K \ge 1$ and $\|\gamma_i\|_{{\cal C}^\alpha}\leq \Lambda_0$, the following estimate holds:
$$
\|\gamma_1-\gamma_2\|_{L^\infty(\Omega)}\leq\,V(\|\Lambda_{\gamma_1}-\Lambda_{\gamma_2}\|_{H^{1/2}(\partial\Omega)\to
H^{-1/2}(\partial\Omega)})
$$
with $V(t)=C\,\log(\frac1t)^{-a}$. Here $C, a>0$ depend
only on $K$, $\alpha$ and $\Lambda_0$.\\
\\
An example, due to Alessandrini \cite{A}, shows that in absence
 of continuity, $L^\infty$ estimates do not hold. Namely, if we
 denote by $B_{r_0}=\{ x \in {\bf R}^2, |x|<r_0\}$ the ball
  centered at the origin with radius $r_0$, take $\Omega= B_{1}$ the unit ball in ${\bf R}^2$, $\gamma_1=1$
  and $\gamma_2= 1+\chi_{B_{r_0}}$, then
  $\| \gamma_1 -\gamma_2\|_{L^\infty(\Omega)}=1$, but
  $\| \Lambda_{\gamma_1} -\Lambda_{\gamma_2} \|_{H^{\frac{1}{2}} \to H^{-\frac{1}{2}}}\leq 2 r_0 \to 0$ as
  $r_0 \to 0$.\\
\\
A closer look to the previous example shows that $\lim_{r_0 \to 0}\|\gamma_1-\gamma_2\|_{L^2(\Omega)}=0$. Therefore one
could conjecture that, in absence of continuity, average stability
(in the $L^2$ sense) might hold. However, it is well known that
some control on the oscillation of $\gamma$ is needed to obtain
stability. Namely, let $\gamma$ be defined in the unit square and
extended periodically. Then the sequence
$\{\gamma(jx)\}_{j=1}^\infty $ $G$-converges to a matrix
$\gamma_0$ (see for example \cite{Tar} for the notion of
$G$-convergence). On one hand, $\gamma( {jx})$ has not any
convergent subsequence in $L^2$. On the  other hand,
$G$-convergence implies the convergence of the fluxes
\cite[Proposition 9]{Tar}. That is, if $u_{j}, u_0$ solve the
corresponding Dirichlet problems for a fixed function $f \in
H^{\frac{1}{2}}(\partial \Omega)$,
\begin{equation}\label{cond2}
\begin{cases}
\nabla\cdot(\gamma_j\nabla u_j)=0\\
u_j{\vline \partial \Omega}=f
\end{cases}
\end{equation}
then, the fluxes satisfy that $\gamma_j u_j \rightharpoonup \gamma
\nabla u$. Thus, by \eqref{weakformulation} $\lim_{j_1,j_2 \to
\infty}
\langle\Lambda_{\gamma_{j_1}}-\Lambda_{\gamma_{j_2}}),\varphi_0\rangle$
for each $\varphi_0$. Notice that $\gamma_j$ can be chosen even being $C^\infty$, so the
problem here is not so much a matter of  regularity but rather a control on the oscillation.\\
\\
In this paper we prove that $L^2$ stability holds if we prescribe a bound of $\gamma$ in any fractional Sobolev space $W^{\alpha,2}$. By the relation with Besov spaces this could be interpreted as controlling the average oscillation of the function. Thus average control on the oscillation of the coefficients yields average stability of the inverse problem.

\begin{Theorem}\label{Thetheorem}
Let $\Omega $ be a Lipschitz domain in the plane. Let $\gamma=\gamma_1, \gamma_2$ be two
planar conductivities in $\Omega$ satisfying
\begin{itemize}
\item{(I)} Ellipticity: $\frac1K\leq \gamma(x)\leq K$.
\item{(II)} Sobolev regularity: $\gamma_i \in W^{\alpha,p}(\Omega)$ with $\alpha>0, 1<p<\infty$, and $\|\gamma_i\|_{{W^{\alpha,p}(\Omega)}}\leq\Gamma_0$.
\end{itemize}
Let $\tilde{\alpha}=\min\{\alpha,\frac{1}{2}\}$. Then there exists two constants $c(K,p)$, $C(K,\alpha,p,\Gamma_0)>0$,  such that:
\begin{equation}
\|\gamma_1-\gamma_2\|_{L^{2}(\Omega)} \leq \frac{C}{|\log(\rho)|^{c\tilde{\alpha}^2}}
\end{equation}
where $\rho=\|\Lambda_{\gamma_1}-\Lambda_{\gamma_2}\|_{H^{1/2}(\partial\Omega) \to H^{-1/2}(\partial\Omega)}$.
\end{Theorem}
\noindent
The theorem is specially interesting for $\alpha \to 0$. Then we are close to get stability for conductivities in  $L^\infty$. \noindent An estimate for the behavior of $C$ in terms of $\Gamma_0$ is obtained in the particular case $p=2$ (see Corollary \ref{finalcoroll}), and analogous results for $p\neq 2$ can be deduced by interpolation. Also interpolating one can obtain $L^p$ stability estimates, whose behavior in $\alpha$ will be quadratic as well.\\
\\
Concerning the logarithmic modulus of continuity, the arguments of Mandache \cite{Ma} can be adapted to the $L^2$ setting. Namely we can consider the same set of conductivities with the obvious replacement of the $C^m$ function by a normalized $W^{\alpha,2}$ function. The argument shows the existence of two conductivities such that $\|\gamma_1-\gamma_2\|_{L^\infty(\D)} \le \epsilon,\, \|\gamma_i\|_{{W^{\alpha,p}(\Omega)}}\leq\Gamma_0$, but
\begin{equation}
\|\gamma_1-\gamma_2\|_{L^2(\D)} \ge \frac{1}{C |\log
(\rho)|^{\frac{3(1+\alpha)}{2\alpha}}}.
\end{equation}
Here $C$ is a constant depending on all the parameters. Notice that the power is better than in the $L^\infty$ setting but still the modulus of continuity is far from being satisfactory.\\
\noindent
In our way to prove Theorem \ref{Thetheorem} we have dealt with several questions related to quasiconformal mappings of independent interest.  More precisely, we have needed to understand how quasiconformal mappings interact with fractional Sobolev spaces. In particular we analyze the regularity of Beltrami equations with Sobolev bounds on the coefficients which has been a recent topic of interest in the theory. See \cite{CFMOZ,CT} where the case $\mu \in W^{1,p}$ is investigated in relation with the size of removable sets. We prove the
following regularity result.

\begin{Theorem}\label{theoremhomeoregularity}
Let $\alpha\in (0,1)$, and suppose that $\mu,\nu\in W^{\alpha, 2}(\C)$ are Beltrami coefficients, compactly supported in $\D$, such that
$$\left||\mu(z)|+|\nu(z)|\right|\leq\,\frac{K-1}{K+1}.$$
at almost every $z\in\D$. Let $\phi:\C\to\C$ be the only homeomorphism satisfying
$$\overline\partial \phi=\mu\,\partial \phi+\nu\,\overline{\partial\phi}$$
and $\phi(z)-z={\cal O}(1/z)$ as $|z|\to\infty$. Then, $\phi(z)-z$ belongs to $W^{1+\theta\alpha, 2}(\C)$ for every $\theta\in(0,\frac{1}{K})$, and
$$
\|D^{1+\theta\alpha}(\phi-z)\|_{L^2(\C)}\leq
C_K\,\left(\|\mu\|_{W^{\alpha,2}(\C)}^\theta+\|\nu\|_{W^{\alpha,2}(\C)}^\theta\right)
$$
for some constant $C_K$ depending only on $K$.
\end{Theorem}
\noindent
Many corolaries can be obtained form this theorem by interpolation, as
for example what do you obtain if $\mu$ is a function of bounded variation. We have contented ourselves with the $L^2$ setting but similar results hold in $L^p$.  As a consequence of
this theorem, we obtain the corresponding regularity of the complex geometric optics solutions.\\
The other  crucial ingredient in our proof  is the regularity of
$\mu \circ \psi$ where $\psi$ is a normalized quasiconformal
mapping. It is well known that quasiconformal mappings preserve
$BMO$ and $\dot{W}^{1,2}$ but it is not clear what happens with
the intermediate spaces. We prove the following stament,
\begin{equation}
\label{comp}
 \mu \in W^{\alpha,2}\hspace{.5cm}\Rightarrow\hspace{.5cm} \mu \circ \psi \in
W^{\beta,2},\hspace{.2cm}\text{ for every }\beta<\frac{\alpha}{K}
\end{equation}
which suffices for our purposes. The proof relies on the fact that Jacobians of quasiconformal mappings are Muckenhoupt
weights \cite{AIS} \\
\\
The Lipschitz regularity of the domain $\Omega$ is used to reduce the problem to the unit disk $\D$. This reduction relies on two facts. First, any Lispchitz domain $\Omega$ is an extension domain for fractional Sobolev spaces. Secondly, the characteristic function $\chi_{\Omega}$ belongs to $W^{\alpha,2}(\C)$ for any $\alpha<\frac12$. Indeed, this is responsible also of the constraint $\tilde{\alpha}<\frac12$ at Theorem \ref{Thetheorem}. In fact, a stability result holds as well if $\Omega$ is any simply connected extension domain. To see this, recall that planar simply connected extension domains $\Omega$ are quasidisks (\cite{Geh}), that is, $\Omega=\phi(\D)$ where $\phi:\C\to\C$ is quasiconformal. Therefore, for instance by our results in Section \ref{beltramifractsob}, $\chi_\Omega=\chi_\D\circ\phi^{-1}$ belongs to some space $W^{\tilde{\alpha},2}$, and then use Theorem \ref{Thetheorem}.\\
\\
The rest of the paper is organized as follows. In Section \ref{strat} we recall previous facts from \cite{AP,BFR} which will be needed in the present paper, and describe the strategy of our proof. In Section \ref{reductionOmega=D} we reduce the problem to conductivities $\gamma$ such that $\gamma-1\in W^{\alpha,2}_0(\D)$. In Section \ref{beltramifractsob} we study the interaction between quasiconformal mappings and fractional Sobolev spaces. Finally in Section \ref{assymptotics} we prove the subexponetial growth of the complex geometric optic solutions and in Section \ref{endofproof} we prove the theorem.\\
\\
In closing we remark several issues raised by our work. The first one is to improve the logarithmic character of the stability. It was proved by Alesssandrini and Vesella that often a logarithmic estimate yields Lipschitz stability for some finite dimensional spaces of conductivities. However, to achieve the desired estimates
in our setting seems to require a more subtle understanding of the Beltrami equation and we leave it for the future. It will also be desirable to obtain $L^p$ estimates in terms of $W^{\alpha,p}$ with constants independent of $p$, so that the ${\cal C}^\alpha$ situation in \cite{BFR} could be understood as a limit of this paper. This seems to require an $L^2$ version of the boundary recovery results of Alessandrini  \cite{A1} and Brown (see \cite{B1}). Finally, from the quasiconformal point of view, there seems to be room for improvement in our estimates specially concerning the composition which is far from being optimal when $\alpha \nearrow 1$, since $\dot{W}^{1,2}$ is invariant under composition with quasiconformal maps. This will also be the issue for further investigations.



\subsection*{Notation}
\noindent Complex and real derivatives are denoted by
$$\aligned
\dc&=\overline\partial=\frac{\partial}{\partial\overline{z}}=\frac12\left(\frac{\partial}{\partial x}+ i\frac{\partial}{\partial y}\right)\\
\d&=\partial=\frac{\partial}{\partial
z}=\frac12\left(\frac{\partial}{\partial x}-
i\frac{\partial}{\partial y}\right)
\endaligned$$
where $z=x+iy$. For a mapping $\phi:\Omega \to \C$, its Jacobian determinant is denoted by $J(z,\phi)=|\d\phi(z)|^2 -|\dc\phi(z)|^2$. For $k \in \C$ we will use the unimodular function $e_{k}(z)=e^{ikz+i\bar k \bar z}$. Notice that then we can define the Fourier transform by
$$\widehat{f}(k)=\int_\C e_{-k}(z)\,f(z)\,dA(z).$$
The spaces $L^p(\Omega)$, $\dot{W}^{1,p}(\Omega)$ and $W^{1,p}(\Omega)$ are defined as usually. Then, following Adams \cite{Ad}, one introduces $W^{\alpha,p}(\Omega)$ as the complex interpolation space
$$W^{\alpha,p}(\Omega)=[L^p(\Omega), W^{1,p}(\Omega)]_\alpha,$$
and similarly for the homogeneous case $\dot{W}^{\alpha,p}(\Omega)=[L^p(\Omega), \dot{W}^{1,p}(\Omega)]_\alpha$.
\noindent
The H\"older space $C^{\alpha}(\Omega)$ over a domain $\Omega$ is
$${\cal C}^\alpha(\Omega)=\left\{f:\|f\|_{L^\infty}+\sup_{x,y \in \Omega}
\frac{|f(x)-f(y)|}{|x-y|^\alpha}<\infty\right\}.$$ For simplicity,
$H^1(\Omega)=W^{1,2}(\Omega)$ and $H^1_0=W^{1,2}_0(\Omega)$. By
$H^\frac{1}{2}(\partial \Omega)$ we denote the quotient space
$H^{1}(\Omega)/ H^1_0(\Omega)$. Given a Banach space $X$ we denote
the operator norm of $T \colon X \to X$ by $\|T\|_X$.
\noindent
We remark that $C$ or $a$ denote constants which may change at each occurrence. We
will indicate the dependence of the constants on parameters $K$, $\Gamma$, etc,
 by writing $C=C(K,\Gamma,...)$.\\
Finally,   for two conductivities $\gamma_1$ and $\gamma_2$, we
write
$$\rho= \|\Lambda_{\gamma_1}-\Lambda_{{\gamma_2}}\|_{H^{1/2} \to
H^{-1/2}}.$$

\subsection*{Acknowledgements}
Part of this work was done in several research visits of A. Clop to the Department of Mathematics of the Universidad Aut\'onoma de Madrid, to which he is indebted for their hospitality. A. Clop is partially supported by projects {\it{Conformal Structures and Dynamics}}, GALA (contract no. 028766), 2005-SGR-00774 (Generalitat de Catalunya) and MTM2007-62817 (Spain). D. Faraco wants to thank C. Sbordone for inspiring remarks concerning $G$ convergence. D. Faraco and A. Ruiz are partially supported by project MTM2005-07652-C02-01 of Ministerio de
Educaci\'on y Ciencia, Gobierno de Espa\~{n}a.


\section{Scheme of the proof}\label{strat}

We will follow the strategy of \cite{BFR}. This work focusses on the approach based on the Beltrami equation iniciated in \cite{AP}. The starting point is the answer to Calder\'on conjecture in the plane obatained by Astala and P\"aiv\"arinta.

\begin{Theorem}[Astala-P\"aiv\"arinta]
Let $\Omega\subset\R^2$ be a bounded simply connected domain, and
let $\gamma_i\in L^\infty(\Omega)$, $i=1,2$. Suppose that there
exist a constant $K>1$ such that $\frac{1}{K}\leq\gamma_i\leq K$.
If
$$\Lambda_{\gamma_1}=\Lambda_{\gamma_2}$$
then $\gamma_1=\gamma_2$.
\end{Theorem}
\noindent
In other words, the mapping $\gamma\mapsto\Lambda_\gamma$ is injective. We recall the basic elements from \cite{AP} needed in the sequel, also the strategies for uniqueness and stability, and what we will need in the current paper.

\paragraph{Equivalence between Beltrami and conductivity equation:}
Let $\D$ be the unit disc. If a function $u$ is $\gamma$-harmonic
in $\D$, then there exists another function $v$, called its
$\gamma$-harmonic conjugate (and actually $\gamma^{-1}$-harmonic
in $\Omega$), unique modulo constants, such that $f=u+iv$
satisfies the $\R$-linear Beltrami type equation
\begin{equation}\label{rlinearbeltrami}
\overline\partial f=\mu\,\overline{\partial f}
\end{equation}
with
\begin{equation}\label{mugamma}
\mu=\frac{1-\gamma}{1+\gamma} \in \R.
\end{equation}
Then if $K \ge 1$ is the ellipticity constant of $\gamma$  we
denote by
\[ \kappa=\frac{K-1}{K+1}.  \]
It is an algebraic fact to show that $\|\mu\|_\infty\le \kappa$ and thus the Beltrami equation is elliptic when so is the conductivity equation and viceversa. Moreover, for $x\in (\frac{1}{K},K)$, the function $F(x)=\frac{1-x}{1+x}$ satisfies $\frac{2}{1+K}\leq|F'(x)|\leq \frac{2K}{1+K}$. Thus, it also follows that
$$
\frac{1}{C}\,\|\gamma\|_{W^{\alpha,p}(\Omega)}\leq\|\mu\|_{W^{\alpha,p}(\Omega)}\leq C\,
\|\gamma\|_{W^{\alpha,p}(\Omega)},
$$
where the constant $C$ only depends on $K$ (see Lemma \ref{KePV}). Therefore, bounds in terms of $\mu$ and $\gamma$ are equivalent. \\
We can argue as well in the reverse direction. If $f\in W^{1,2}_{loc}(\D)$ satisfies (\ref{rlinearbeltrami}) for real $\mu$ with $\|\mu\|_\infty\leq\kappa$, then we can write $f=u+iv$ where $u$ and
$v$ satisfy
$$\div\left(\gamma\,\nabla u\right)=0\hspace{1cm}\text{and}\hspace{1cm}\div\left(\gamma^{-1}\,\nabla
v\right)=0.
$$
Thus, it is equivalent to determine either $\gamma$ or $\mu$, and throughout the paper we will work with either of them indistinctly.\\
\\
As for holomorphic functions, $u$ and $v$ are related by the corresponding Hilbert transform
$${\cal H}_\mu \colon H^\frac{1}{2}(\partial\D)\to H^\frac{1}{2}(\partial\D)$$
defined as
$${\cal H}_\mu(u|_{\partial\D})=v|_{\partial\D}$$
for real functions, and $\R$-linearly extended to $\C$-valued
functions by setting ${\cal H}_\mu(iu)=i\,{\cal H}_{-\mu}(u)$.
Since $\partial_T {\cal H}_\mu=\Lambda_\gamma$ it follows
\cite[Proposition 2.7]{AP} that ${\cal H}_\mu$, ${\cal H}_{-\mu}$
and $\Lambda_{\gamma^{-1}}$ are uniquely determined by
$\Lambda_\gamma$. Accordingly in \cite[Proposition 2.2]{BFR} it is
shown that
$$\|{\cal H}_{\mu_1}-{\cal H}_{\mu_2}\|\lesssim\|\Lambda_{\gamma_1}-\Lambda_{\gamma_2}\|$$
with respect to the corresponding operator norms. In other words, the mapping $\Lambda_\gamma\mapsto{\cal H}_\mu$ is Lipschitz continuous independently of the regularity of $\gamma$.

\paragraph{Existence of complex geometric optics solutions, scattering transform and $\partial_k$ equations:}
The theory of quasiconformal mappings and Beltrami operators allow to combine in an efficient way ideas from complex analysis, singular integral operators and degree arguments to prove the existence of \emph{complex geometric optics solutions} with no assumptions on the coefficients.

\begin{Theorem}\label{existuniqbeltrami}
Let $\kappa\in(0,1)$, and let $\mu$ be a real Beltrami coefficient, compactly supported in $\D$, satisfying $\|\mu\|_\infty<\kappa$. For every $k\in\C$ and $p\in(2,1+\frac{1}{\kappa})$ the equation
$$
\overline\partial f=\mu\,\overline{\partial f}
$$
admits a unique solution $f\in W^{1,p}_{loc}(\C)$ of the form
$$
f(z)=e^{ikz}\,M_\mu(z,k)
$$
such that $M_\mu(z,k)-1={\cal O}(1/z)$ as $|z|\to\infty$. Moreover,
$$\Re\left(\frac{M_{-\mu}}{M_\mu}\right)>0$$
and $f_\mu(z,0)=1$.
\end{Theorem}
\noindent
\noindent
In this context, the proper definition of scattering transform of $\mu$ (or of $\gamma$) is
\begin{equation}
\tau_\mu(k)=\frac{i}{4\pi}\int_{\D}\frac{\partial}{\partial z} \left(e^{i\overline{kz}}(\overline{f_\mu(z)}-\overline{f_{-\mu}(z)})\right)\,dA(z).
\end{equation}
The complex geometric optics solutions $\left\{u_\gamma, \tilde{u}_\gamma\right\}$ to the divergence type equation \eqref{cond} are then obtained from the corresponding ones from the Beltrami equation by
$$\aligned
u_\gamma&=\Re(f_\mu)+i\Im(f_{-\mu})\\
\tilde{u}_\gamma&=\Im(f_{\mu})+i\Re(f_{-\mu}),\\
\endaligned$$
and they uniquely determine the pair $\left\{f_\mu,f_{-\mu}\right\}$ (and viceversa) in a stable way. We consider $u_\gamma$ as a function of $(z,k)$. In the $z$ plane, $u_\gamma$ satisfies the complex $\gamma$-harmonic equation, $$\div(\gamma\,\nabla u_\gamma)=0.$$
As a function of $k$, $u_\gamma$ is a solution to the following $\overline\partial$-type equation \begin{equation}\label{dbarequation}
\frac{\partial
u_\gamma}{\partial\overline{k}}(z,k)=-i\,\tau_\mu(k)\,\overline{u(z,k)}.
\end{equation}
Let us emphasize that $\tau_\mu(k)$ is independent of $z$.

\paragraph{Strategy for uniqueness:} Let $\gamma_1,\gamma_2$ be two conductivities. In \cite{AP}, the strategy for uniqueness is divided in the following steps:
\begin{description}
\item [$(i)$] Reduction to $\D$. 
\item [$(ii)$] If $\Lambda_{\gamma_1}=\Lambda_{\gamma_2}$, then $\tau_{\mu_1}=\tau_{\mu_2}$. 
\item [$(iii)$] Step $(ii)$ and (\ref{dbarequation})
imply that $u_{\gamma_{1}}=u_{\gamma_{2}}$. \item [$(iii)$]
Finally, condition $u_{\gamma_{1}}=u_{\gamma_2}$ is equivalent to
$Du_{\gamma_1}=Du_{\gamma_2}$, which holds as well if and only if
$\gamma_1=\gamma_2$
\end{description}
\noindent
First step is relatively easy since there is no regularity of $\gamma$ to preserve and thus one can extend by $0$ in $\D\setminus \Omega$. Second step is dealt with in \cite[Proposition 6.1]{AP}. It is shown that ${\cal H}_{\mu_1}={\cal H}_{\mu_2}$ implies $f_{\mu_1}(z,k)=f_{\mu_2}(z,k)$ for all $k\in\C$ and $|z|>1$. As a consequence $(ii)$ follows.\\
\\
The step $(iii)$ is more complex because uniqueness results and a priori estimates for pseudoanalytic equations in $\C$ like (\ref{dbarequation}) only hold if the coefficients or the solutions decay fast enough at $\infty$. Unfortunately the required decay properties for $\tau$ seem to require roughly one derivative for $\gamma$. However in \cite{AP} it is shown that in the measurable setting at least we obtain subexponential decay. That is, we can write,
\begin{equation}\label{decay}
u_{\gamma}(z,k)=e^{ik\left(z+\epsilon_\mu(z,k)\right)}
\end{equation}
for some function $\epsilon=\epsilon_\mu(z,k)$ satisfying
$$\lim_{k \to \infty} \|\epsilon_\mu (z,k)\|_{L^\infty(\C)}=0.$$
This would not be enough if we would consider equation (\ref{dbarequation}) for a single $z$. However, in \cite{AP} it is used that $u(z,k)$ solves an equation for each $z$. Further, one has asymptotic estimates for $u$ both in the $k$ (as above) and $z$ variables. Then, a clever topological argument in both variables shows that, with these estimates, $\tau_{\mu}$ determines the solution to (\ref{dbarequation}).

\paragraph{Strategy for stability:}In order to obtain stability, the natural
idea is to try to quantify in an uniform way the arguments for uniqueness. This was done
in \cite{BFR} for ${\cal C}^\alpha$ conductivities. Let us recall the argument and
specially the results which did not require regularity of $\gamma$ and would be
 instrumental for the current work. Let $\rho=\|\Lambda_{\gamma_1}-\Lambda_{\gamma_2}\|$. First
 one reduces to the unit disk by an argument which involves Whitney extension operator,
  the weak formulation (\ref{weakformulation}) and a result of Brown about recovering continuous
conductivities at the boundary (\cite{B1}). Next we investigate the
 relation between the corresponding  scattering transforms.

\begin{Theorem}[Stability of the scattering transforms]\label{stabilityD}
Let $\gamma_1,\gamma_2$ be conductivities in $\D$, with $\frac{1}{K}\leq\gamma_i\leq K$, and denote $\mu_i=\frac{1-\gamma_i}{1+\gamma_i}$. Then, for every $k \in \C$ it holds that
\begin{equation}
|\tau_{\mu_1}(k)-\tau_{\mu_2}(k)|\leq c\,e^{c|k|}\, \rho.
\end{equation}
where the constant $c$ depends only on $K$.
\end{Theorem}
\noindent The estimate is just pointwise but on  the positive side it holds for $L^{\infty}$ conductivities. In \cite[Theorem 4.6]{BFR} it is also given an explicit formula for the difference of scattering transforms which might be of  independent interest. Next we state a result that is implicitely proved in \cite[Theorem 5.1]{BFR}. There it is stated as a property of solutions to regular conductivities. However, in the proof the regularity  is only used to obtain the  decay in the $k$ variable. Because of this, here we state it separately as condition $(\ref{decayy})$.

\begin{Theorem}[A priori estimates in terms of scattering transform]\label{teoremaBFR}
Let $K\ge 1$ and  $\gamma_1$, $\gamma_2$ be conductivities on $\D$, with $\frac1K\leq\gamma_i\leq K$. Let
$$u_{\gamma_j}(z,k)=e^{ik\left(z+\epsilon_{\mu_j}(z,k)\right)},$$
denote, as in \eqref{decay}, the complex geometric optics solutions to \eqref{cond}. Let us assume that there exist positive constants $\alpha, B$ such that for eack $z,k\in\C$,
\begin{equation}\label{decayy}
 |\epsilon_{\mu_i}(z,k)|\leq \frac{B}{|k|^{\alpha}}.
\end{equation}
Then it follows that:
\begin{description}
\item [\textbf{A}] There exists new constants $b=b(K)$, $C=C(K,B)$, such that for every $z\in \C$ there exists $w\in\C$ satisfying:
\begin{itemize}
\item[(a)] $|z-w|\leq C B\,\left|\log\frac1\rho\right|^{-b\alpha}$, where $\rho=\|\Lambda_{\gamma_1}-\Lambda_{\gamma_2}\|$.
\item[(b)] $u_{\gamma_1}(z,k)=u_{\gamma_2}(w,k)$.
\end{itemize}
\item [\textbf{B}] For each $k \in \C$, there exists new constants $b=b(K)$ and $C=C(k,K)$ such that
\begin{equation}\label{linfty}
\|u_{\gamma_1}(z,k)-u_{\gamma_2}(z,k)\|_{L^{\infty}(\D,dA(z))}
\le\frac{C B^\frac1K}{|\log(\rho)|^{b\alpha}}.
\end{equation}
\end{description}
\end{Theorem}
\begin{proof}
The proof of \textbf{A} follows from \cite[Proposition 5.2]{BFR} and \cite[Proposition 5.3]{BFR}. Let us prove \textbf{B}. Given $z\in\C$, let  $w\in\C$ be given by part \textbf{A}. Then
$$|u_{\gamma_1}(z,k)-u_{\gamma_2}(z,k)|=|u_{\gamma_1}(z,k)-u_{\gamma_1}(w,k)|.$$
By the H\"older continuity of $K$-quasiregular mappings, together with $(a)$, we get
$$|u_{\gamma_1}(z,k)-u_{\gamma_2}(z,k)|\leq C(k,K)\, |z-w|^\frac1K \leq C(k,K)\,C^\frac1K\,B^\frac1K\left|\log\frac1\rho\right|^{-\frac{b\alpha}{K}}$$
and the desired estimate follows after renaming the constants.
\end{proof}

\noindent

\noindent
Unlike in the uniqueness arguments, going from $u_{\gamma_1}-u_{\gamma_2}$ to $D(u_{\gamma_1}-u_{\gamma_2})$ is more delicated in the stability setting, since functions do not control their derivatives in general. This is solved in \cite{BFR}, under H\"older regularity, using the following fact.

\begin{Theorem}[Schauder estimates]
Let $\gamma_i$, $i=1,2$ be conductivities on $\D$, such that $\frac1K\leq \gamma_i\leq K$ and $\|\gamma_1\|_{{\cal C}^\alpha(\D)}\le \Gamma_0$. As always, denote $\mu_i=\frac{1-\gamma_i}{1+\gamma_i}$, and let $f_{\mu_i}(z,k)$ be the corresponding complex geometric optics solutions to (\ref{rlinearbeltrami}). Then
\begin{enumerate}
\item For each $k\in\C$ there esists a constant $C=C(k)>0$ with
\begin{equation}\label{firstBFR}
\|f_{{\mu_1}}(\cdot, k)-f_{{\mu_2}}(\cdot, k)\|_{{\cal
C}^{1+\alpha}(\D)}\leq C(k).
\end{equation}
\item The jacobian determinant of $f_{\mu_i}(z,k)$ has a positive lower bound
$$J(z,f_{\mu_i}(\cdot,k)) \ge C(K,k,\Gamma_0).$$
\end{enumerate}
\end{Theorem}
\noindent
Now, to finish the proof of stability for H\"older continuous conductivities,
just note that an interpolation argument between $L^\infty$ and $C^{1+\alpha}$ gives Lipschitz bounds for $Df_{\mu_i}$. Thus, by $\mu=\frac{\overline{\partial}f}{\overline{\partial f}}$ and the second statement above, one obtains $L^\infty$ stability for $\mu_1-\mu_2$. The corresponding result for $\gamma_1-\gamma_2$ comes due to (\ref{mugamma}).

\paragraph{Strategy for stability under Sobolev regularity} In the current work
we will try to push the previous strategy to obtain $L^2$ stability. The previous analysis shows that we can rely in many of the results from \cite{AP,BFR}. In particular, we only have to prove that $\tau_\mu\mapsto\mu$ is continuous.\\
For this, we start by reducing the problem in Section \ref{reductionOmega=D}. We replace the assumption $\gamma_i\in W^{\alpha,p}(\Omega)$ by $\gamma_i\in W^{\beta,2}_0(\D)$, where $0<\beta<\min\{\frac12,\alpha\}$. For this, it is used there that characteristic functions of Lipschitz domains belong to $W^{\beta,q}(\C)$ whenever $\beta q<1$. \\
Then we follow by investigating the regularity of solutions of Beltrami equations with coefficients in fractional Sobolev spaces in order to obtain an estimate like (\ref{firstBFR}), with the ${\cal C}^{1+\alpha}$ norm replaced by the sharp Sobolev norm attainable under our assumption on the Beltrami coefficient (see Theorem \ref{regcgos}). It is also needed here to understand how composition with quasiconformal mappings affects fractional Sobolev spaces. As far as we know, the estimates here are new and of their own interest.\\
Afterwards we prove that our Sobolev assumption on $\mu$ suffices to get the uniform subexponential growth  of the geometric optics solutions needed in condition $(\ref{decayy})$ in Theorem~\ref{teoremaBFR} (this is done in Section \ref{assymptotics}, see Theorem \ref{nonlineardecay}). In fact we obtain a very clean   expression for the precise growth, achieving that the exponent depends linearly on  $\alpha$. Finally, in Section \ref{endofproof} we do the interpolation argument. Here we do not have enough regularity to control $W^{1,\infty}$ norms and
here is where one sees why we need to be happy with the control on $\|\mu_1-\mu_2\|_{L^2(\D)}$. Also we do not have a pointwise lower bound for the corresponding Jacobians which causes also difficulties.


\section{Fractional Sobolev spaces and Reduction to $\mu\in W^{\alpha,2}_0(\D)$}\label{reductionOmega=D}

 \subsection{On fractional Sobolev Spaces}

Following \cite[p.21]{Ad}, for any domain $\Omega$, we denote by
$W^{1,p}(\Omega)$ the class of $L^p(\Omega)$ functions $f$ with
$L^p(\Omega)$ distributional derivatives of first order. This
means that for any constant coefficients first order differential
operator $D$ there exists an $L^2(\Omega)$ function $Df$ such that
$$
\int_\Omega f\,D\varphi=-\int_\Omega Df\,\varphi
$$
whenever $\varphi\in {\cal C}^\infty$ is compactly supported
inside of $\Omega$. Similarly one can
define the Sobolev spaces $W^{m,p}(\Omega)$ of general integer order $m\geq 1$.\\
\\
It comes from the work of Calder\'on (see \cite[p.7]{AH} or \cite{St}) that every Lipschitz domain $\Omega$ is an \emph{extension domain}. That is, for any integer $m>0$ and any domain $\Omega'\supset\overline{\Omega}$ there exists a bounded linear extension operator
$$E_m:W^{p,2}(\Omega)\to W^{m,p}_0(\Omega')$$
and therefore for every function $f\in W^{m,2}(\Omega)$ there is another function $E_mf\in W^{m,p}(\Omega')$ such that $E_mf_{|\Omega}=f$. Of course, $E_mf\in W^{1,p}(\C)$.\\
\\
Let us introduce for general domains $\Omega$ and any real number $0<\alpha<1$ the complex interpolation space
$$
W^{\alpha,p}(\Omega)=[L^p(\Omega), W^{1,p}(\Omega)]_\alpha.
$$
The closure of ${\cal C}^\infty_0(\Omega)$ (${\cal C}^\infty$ functions with compact support contained in $\Omega$) in $W^{\alpha, p}(\Omega)$ is denoted by $W^{\alpha, p}_0(\Omega)$. Functions in $W^{\alpha,p}_0(\Omega)$ can be extended by zero to the whole plane, and the extension belongs to $W^{\alpha,p}(\C)$. Thus, we can identify any function in $W^{\alpha, p}_0(\Omega)$ with its extension in $W^{\alpha,p}(\C)$.\\
\\
When $\Omega$ is an extension domain, an interpolation argument shows (see \cite[p.222]{Ad}) that $W^{\alpha,p}(\Omega)$ coincides with the space of restrictions to $\Omega$ of functions in $W^{\alpha,p}(\C)$. That is, to each function $u\in W^{\alpha,p}(\Omega)$ one can associate a function $\tilde{u}\in W^{\alpha,p}(\C)$ such that $\tilde{u}_{|\Omega}=u$ and $\|\tilde{u}\|_{W^{\alpha,p}(\C)}\leq  C\,\|u\|_{W^{\alpha,p}(\Omega)}$. \\
\noindent
We have chosen just one way to introduce the fractional Sobolev spaces. In the rest of the subsection, we discuss the alternative characterizations and properties of these spaces needed in the rest of the paper. Two good sources for the basics of this theory are \cite[Chapter 7]{Ad}, \cite[Chapter 4]{St}.

\paragraph{Fourier side.}
For $p=2$, it is easy to see that,
$$
W^{\alpha,2}(\C)=\left\{f\in L^2(\C);
(1+|\xi|^2)^\frac{s}{2}\,\widehat{f}(\xi)\in L^2(\C)\right\}
$$
and that this agrees with the space of Bessel potentials
$$
W^{\alpha,2}(\C)=\left\{f=G_\alpha\ast g; g\in L^2(\C)\right\}
$$
where $G_\alpha$ is the Bessel Kernel \cite[p.10]{AH}. For $p\neq 2$ the situation is more complicated but it can be shown that
$$
W^{\alpha,p}(\C)=\left\{f\in L^p(\C); \left((1+|\xi|^2)^\frac{s}{2}\,\widehat{f}(\xi)\right)^\wedge \in L^p(\C)\right\}.$$

\paragraph{Integral modulus of continuity}
We define the $L^p$-difference of a function $f$ by
\begin{equation}
\omega_p(f)(y)=\|f(\cdot+y)-f(\cdot)\|_{L^p(\C)}.
\end{equation}
( see  \cite[Chapter V]{St} Then the Besov spaces $B^{p,q}_\alpha(\R^n)$ are defined by
$$B^{p,q}_\alpha(\R^n)=\{f \in L^p(\R^n): \int_{\R^n}  \omega_p(f)(y)^q |y|^{-(n+\alpha
q)}<\infty.$$
There are many relations between Besov and fractional Sobolev spaces. We will need the following two facts,
\begin{equation}\label{BesovSobolev}
B^{2,2}_\alpha=W^{\alpha,2},\hspace{2cm}
W^{\alpha,p} \subset B^{p,2}_\alpha \quad (p<2).
\end{equation}
For a proof see \cite[Chapter 7]{Ad} or \cite[Chapter V]{St}.

\paragraph{Leibniz Rule}[\cite{KePoVe}]
\begin{Lemma}\label{KePV} Let $\alpha\in(0,1)$ and $p\in(1,\infty)$.
\begin{enumerate}
\item[(a)] Let $f,g\in {\cal C}^\infty_0(\C)$. Then,
$$
\|D^\alpha(fg)-f\,D^\alpha(g)-g\,D^\alpha(f)\|_p\leq
C\,\|D^{\alpha_1}(f)\|_{p_1}\,\|D^{\alpha_2}(g)\|_{p_2}
$$
whenever $\alpha_1,\alpha_2\in[0,\alpha]$ are such that
$\alpha_1+\alpha_2=\alpha$ and $p_1,p_2\in(1,\infty)$ satisfy
$\frac{1}{p_1}+\frac{1}{p_2}=\frac{1}{p}$. \item[(b)] Let
$f,g\in{\cal C}^\infty_0(\C)$. Then
$$
\|D^\alpha(f\circ g)\|_p\leq C\,\|Df(g)\|_{p_1}\,\|D^\alpha
g\|_{p_2}
$$
whenever $p_1,p_2\in(1,\infty)$ satisfy
$\frac{1}{p_1}+\frac{1}{p_2}=\frac{1}{p}$. \item[(c)] Let
$f,g\in{\cal C}^\infty_0(\C)$. Then
$$
\|D^\alpha(fg)-f\,D^\alpha(g)-g\,D^\alpha(f)\|_p\leq
C\,\|D^{\alpha}(f)\|_p\,\|g\|_{\infty}
$$
whenever $0<\alpha<1$ and $1<p<\infty$.
\end{enumerate}
\end{Lemma}

\begin{Rem}\label{LebnitzHolder}
From property (a) and (c) it follows the generalized Leibniz rule
\begin{equation}
\|D^\alpha(f\,g)\|_p\leq\|D^\alpha
f\|_{p_1}\,\|g\|_{p_2}+\|D^\alpha g\|_{p_3}\,\|f\|_{p_4}
\end{equation}
whenever $1\leq p_1,p_2,p_3,p_4\leq\infty$ and
$\frac{1}{p}=\frac{1}{p_1}+\frac{1}{p_2}=\frac{1}{p_3}+\frac{1}{p_4}$.
Moreover if $\textrm{suppt}(f,g) \in \D$ we have that
\begin{equation}\label{KPVcompactsupport}
\|D^\alpha(f\,g)\|_p\leq\|D^\alpha
f\|_{p_1}\,\|g\|_{p_2}+\|D^\alpha
g\|_{L^{p_3}(\D)}\,\|f\|_{L^{p_4}(\D)}
\end{equation}
\end{Rem}

\paragraph{Pointwise Inequalities}

\begin{Lemma}[Pointwise inequalities, \cite{Sw}]
If $f\in W^{\alpha,p}(\C)$, $\alpha>0$, $1<p<\infty$, then for
each $0<\lambda<\alpha$ there exists a function $g=g_\lambda\in
L^{p_\lambda}(\C)$, $p_\lambda=\frac{2p}{2-(\alpha-\lambda)p}$
such that
\begin{equation}\label{pointwise}
|f(z)-f(w)|\leq |z-w|^\lambda\,\left(g(z)+g(w)\right)
\end{equation}
for almost every $z,w\in\C$. Furthermore, we have that
$$\|g\|_{L^{p_\lambda}(\C)}\leq\,C_\lambda\,\|f\|_{W^{\alpha,p}(\C)},$$
and the constant $C_\lambda$ remains bounded as
$\lambda\to\alpha$.
\end{Lemma}

\subsection{Reduction to $p=2$}\label{pequals2}

This reduction relies on the fact that $\mu \in L^\infty(\C) \cap W^{\alpha,p}(\C)$ and the following interpolation Lemma.

\begin{Lemma}
Let $f\in W^{\alpha_0,p_0}\cap W^{\alpha_1,p_1}$, where $1< p_0,p_1<\infty$, $0\leq\alpha_0,\alpha_1\leq 1$, and $\theta\in(0,1)$. Then,
$$
\|f\|_{W^{\alpha,p}}\leq\|f\|_{W^{\alpha_0,p_0}}^\theta\,\|f\|_{W^{\alpha_1,p_1}}^{1-\theta}
$$
where
$$\alpha=\theta\,\alpha_0+ (1-\theta)\,\alpha_1\hspace{.5cm}\text{ and }\hspace{.5cm}\frac{1}{p}=\frac{\theta}{p_0}+\frac{1-\theta}{p_1}.$$
Furthermore, if either $p_0=\infty$ or $p_1=\infty$, then the above inequality holds true by replacing $W^{\alpha_i,p_i}$ by the Riesz potentials space $I_{\alpha_i}\ast BMO$.
\end{Lemma}
\begin{proof}
It is well known that the complex interpolation method gives $$[W^{\alpha_0,p_0},W^{\alpha_1,p_1}]_\theta=W^{\alpha,p}$$
whenever $1<p<\infty$ (for the proof of this, see for instance \cite{Tr}). For $p=\infty$, the same result holds true if we replace $W^{\alpha,\infty}$ by the space of Riesz potentials $I_\alpha\ast BMO$ of $BMO$ functions (for this, see \cite{RR}).
\end{proof}
\noindent
Let $\mu$ be a compactly supported Beltrami coefficient. Then, it belongs both to $L^1(\C)$ and $L^\infty(\C)$. If we also assume that $\mu\in W^{\alpha,p}(\C)$ for some $\alpha, p$, then we can use the above interpolation to see that $\mu\in W^{\beta,q}(\C)$, for any $1<q<\infty$ and some $0<\beta<\alpha$. We are particularly interested in $q=2$.

\begin{Lemma}\label{pneq2}
Suppose that $\mu\in W^{\alpha, p}(\Omega)\cap L^\infty(\Omega)$ for some $p>1$ and $0<\alpha<1$. Then,
\begin{itemize}
\item For any $0\leq\theta\leq 1$, $$\|\mu\|_{W^{\alpha\theta,\frac{p}{\theta}}(\Omega)}\lesssim\,\|\mu\|_{L^\infty(\Omega)}^{1-\theta}\,\|\mu\|_{W^{\alpha, p}(\Omega)}^{\theta}.$$
\item For any $0\leq\theta\leq 1$, $$\|\mu\|_{W^{\theta\alpha,\frac{p}{(1-\theta)p+\theta}}(\Omega)}\lesssim\,\|\mu\|_{L^1(\Omega)}^{1-\theta}\,\|\mu\|_{W^{\alpha, p}(\Omega)}^{\theta}.$$
\item One always has
$$\|\mu\|_{W^{\beta,2}(\Omega)}\leq C(K,p)\,\|\mu\|_{W^{\alpha,p}(\Omega)}^{p^\ast/2},$$
where $\beta=\frac{\alpha p^\ast}{2}$ and $p^\ast=\min\{p,\frac{p}{p-1}\}$.
\end{itemize}
\end{Lemma}
\begin{proof}
The first inequality comes easily interpolating between $BMO(\Omega)$ and $W^{\alpha, p}(\Omega)$ (see \cite{RR} for more details). For the second, simply notice that compactly supported Beltrami coefficients belong to all $L^p(\Omega)$ spaces, $p>1$, so one can do the same between $L^{1+\varepsilon}(\Omega)$ ($\varepsilon$ as small as desired) and $W^{\alpha, p}(\Omega)$. The last statement is obtained by letting $\theta=\frac{p^\ast}{2}$ above.
\end{proof}

\subsection{Reduction to $\Omega=\D$ and $\mu\in W^{\alpha, p}_0(\D)$}

The proof of the following lemma relies in the fact that
characteristic functions of Lipschitz belong to $W^{\alpha,2}$ for
each $\alpha<\frac{1}{2}$.

\noindent

\begin{Theorem}\label{extensionzero}
Let $\Omega$ be a Lipschitz domain, strictly included in $\D$. Let $\mu\in W^{\alpha,2}(\Omega)$. Define
$$\tilde{\mu}=\begin{cases}
\mu&\Omega\\0&\C\setminus\Omega
\end{cases}.$$
Then, $\tilde{\mu}\in W^{\beta,2}_0(\C)$ for $\beta<\min\{\alpha,\frac12\}$ and $$\|\tilde{\mu}\|_{W^{\beta,2}(\C)}\leq
C\,\|\mu\|_{W^{\alpha,2}(\C)}.$$
Analogous results can be stated for the extensions by $1$ of $\gamma_i$.
\end{Theorem}
\begin{proof}
Since $\Omega$ is an extension domain, there is an extension $\mu_0$ of $\mu$ belonging to $W^{\alpha,2}(\C)$.
Of course, such extension $\mu_0$ need not be supported in $\Omega$ any more. Now
$\tilde{\mu}$ can be introduced as the pointwise multiplication
$$\tilde{\mu}=\chi_\Omega\,\mu_0.$$
By virtue  Lemma \ref{KePV} it is enough to  study the smoothness
of the characteristic function $\chi_\Omega$. A way to see this is
to  recall that fraccional Sobolev spaces are invariant under
composition with bilipschitz maps \cite{Z}. Now, the
characteristic function of the half plane belongs to
$W^{\alpha,p}_{loc}(\C)$ whenever $\alpha p<1$. Therefore, by a
partition of unity argument, we get that $\chi_\Omega\in
W^{\alpha,p}(\C)$ when $\alpha p<1$. The proof is conclude.
\end{proof}

\noindent
Now we need to compare the original Dirichlet-to-Neumann maps with the Dirichlet-to-Neumann maps of the extensions.

\begin{Lemma}\label{DtNstability}
Let $\Omega$ be a domain strictly included in $\D$. Let $\gamma_1,\gamma_2\in L^\infty(\Omega)$ be  conductivities in $\Omega$. Further, assume that
$$
\frac{1}{K}\leq \gamma_i(z)\leq K
$$
for almost every $z\in\Omega$. Let $\tilde{\gamma_i}$ denote
the corresponding extensions by $1$ to all of $\C$. Then,
$$
\|\Lambda_{\tilde{\gamma_1}}-\Lambda_{\tilde{\gamma_2}}\|_{H^\frac12(\partial\D)\to
H^{-\frac12}(\partial\D)}\leq
C\,\|\Lambda_{\gamma_1}-\Lambda_{\gamma_2}\|_{H^{\frac12}(\partial\Omega)\to
H^{-\frac12}(\partial\Omega)}.
$$
\end{Lemma}
\begin{proof}
We follow the ideas of \cite[Theorem 6.2]{BFR}, although the
stability result from \cite{B1} is not needed in our situation.
Let $\varphi_0\in H^\frac12(\partial\D)$. Let $\tilde{u}_j\in
H^1(\D)$ be the solution to
$$
\begin{cases}
\nabla\cdot(\tilde{\gamma}_j\nabla\tilde{u}_j)=0&\text{ in }\D\\
\tilde{u}_j=\varphi_0&\text{ in }\partial\D.
\end{cases}
$$
Let also $u_2$ be defined by
$$
\begin{cases}
\nabla\cdot(\gamma_2\nabla u_2)=0&\text{ in }\Omega\\
u_2=\tilde{u}_1&\text{ in }\partial\Omega.
\end{cases}
$$
Define now
$\tilde{v}_2=u_2\,\chi_{\Omega}+\tilde{u}_1\,\chi_{\D\setminus\Omega}$.
As in \cite{BFR}, we first control $\tilde{u}_2-\tilde{v}_2$ in
terms of $\rho$. To do this,
$$\aligned
\int_\D|\nabla(\tilde{v}_2-\tilde{u}_2)|^2
&\leq c\,\int_\D\tilde{\gamma}_2\, \nabla(\tilde{v}_2-\tilde{u}_2)\cdot\nabla(\tilde{v}_2-\tilde{u}_2)\\
&=
c\,\int_\D\tilde{\gamma}_2\, \nabla\tilde{v}_2\cdot\nabla(\tilde{v}_2-\tilde{u}_2)\endaligned
$$
because $\tilde{v}_2-\tilde{u}_2\in H^1_0(\D)$ and the
$\tilde{\gamma}_2$-harmonicity of $\tilde{u}_2$ in $\D$. By adding
and substracting $\int_\D  \tilde{\gamma}_1  
\nabla\tilde{u}_1\cdot\nabla(\tilde{v}_2-\tilde{u}_2)$, and
using
 that $\tilde{\gamma}_1=\tilde{\gamma}_2=1$ off $\Omega$, the right
hand side above is bounded by a constant times
$$
\aligned
\left|\int_\D\tilde{\gamma}_1 \nabla\tilde{u}_1\cdot\nabla(\tilde{v}_2-\tilde{u}_2)\right|+
\left|\int_{\Omega} ( \gamma_1\,\nabla\tilde{u}_1-\gamma_2\,\nabla
u_2)\cdot\nabla(\tilde{v}_2-\tilde{u}_2)\right|.
\endaligned
$$
Here the first term vanishes because $\tilde{u}_1$ is
$\tilde{\gamma}_1$-harmonic on $\D$ and
$\tilde{v}_2-\tilde{u}_2\in H^1_0(\D)$. For the second, we observe
that $\tilde{u}_1$ is $\gamma_1$-harmonic in $\Omega$, $u_2$ is
$\gamma_2$-harmonic in $\Omega$, and $u_2-\tilde{u}_1\in
H^1_0(\Omega)$. Thus,
$$\aligned
\left|\int_{\Omega} (\gamma_1\,\nabla\tilde{u}_1-\gamma_2\,\nabla
u_2)\cdot\nabla(\tilde{v}_2-\tilde{u}_2)\right| &=
\left| \langle(\Lambda_{\gamma_1}-\Lambda_{\gamma_2})(\tilde{u}_{1|\partial\Omega}),(\tilde{v}_2-\tilde{u}_2)_{|\partial\Omega}\rangle\right|\\
&\leq\rho\,\|\tilde{u}_1\|_{H^\frac12(\partial\Omega)}\,\|\tilde{v}_2-\tilde{u}_2\|_{H^\frac12(\partial\Omega)}\\
&\leq\rho\,\|\nabla \tilde{u}_1\|_{L^2(\Omega)}\,\|\nabla(\tilde{v}_2-\tilde{u}_2)\|_{L^2(\Omega)}\\
\endaligned$$
Summarizing, we get
\begin{equation}\label{eq4}
\aligned
\left(\int_\D|\nabla(\tilde{v}_2-\tilde{u}_2)|^2\right)^\frac12
&\leq c\,\rho\,\|\nabla \tilde{u}_1\|_{L^2(\Omega)}\leq c\,\rho\,\|\nabla \tilde{u}_1\|_{L^2(\D)}\\
&\leq c\,\rho\,\|\varphi_0\|_{H^\frac12(\partial\D)}.
\endaligned
\end{equation}
We will use this to compare the Dirichlet-to-Neumann maps at
$\partial\D$. If $\psi_0\in H^\frac12(\partial\D)$ is any testing
function, and $\psi$ is any $H^1(\D)$ extension,
\begin{equation}\label{eq5}
\langle(\Lambda_{\tilde{\gamma}_1}-\Lambda_{\tilde{\gamma}_2})(\varphi_0),\psi_0\rangle=\int_\D(\tilde{\gamma}_1\,\nabla\tilde{u}_1-\tilde{\gamma}_2\,\nabla\tilde{u}_2)\cdot\nabla\psi.
\end{equation}
We will divide the bound of this quantity in two steps. For the
first,
$$
\left|\int_\D(\tilde{\gamma}_1\,\nabla\tilde{u}_1-(\gamma_2\,\chi_{\Omega}+\tilde{\gamma_1}\,\chi_{\D\setminus\Omega})\,\nabla\tilde{v}_2)\cdot\nabla\psi\right|
=\left|\langle(\Lambda_{\gamma_1}-\Lambda_{\gamma_2})(\tilde{u}_{1|\partial\Omega}),\psi_{|\partial\Omega}\rangle\right|
$$
which is bounded by
$$
\aligned
\rho\,\|\tilde{u}_{1|\partial\Omega}\|_{H^\frac12(\partial\Omega)}\,\|\psi_{|\partial\Omega}\|_{H^\frac12(\partial\Omega)}
&\leq\rho\,\|\nabla\tilde{u}_1\|_{L^2(\Omega)}\,\|\nabla\psi\|_{L^2(\Omega)}\\
&\leq\rho\,\|\nabla\tilde{u}_1\|_{L^2(\D)}\,\|\nabla\psi\|_{L^2(\D)}\\
&\leq\rho\,\|\varphi_0\|_{H^\frac12(\partial\D)}\,\|\psi_0\|_{H^\frac12(\partial\D)}.
\endaligned$$
We are left with,
$$
\left|\int_\D\left((\gamma_2\,\chi_{\Omega}+\tilde{\gamma_1}\,\chi_{\D\setminus\Omega})\nabla\tilde{v}_2-\tilde{\gamma}_2\,\nabla\tilde{u}_2\right)\cdot\nabla\psi\right|
$$
which  is smaller than,
$$
\aligned
\left|\int_\Omega\gamma_2\,\nabla(\tilde{v}_2-\tilde{u}_2)\cdot\nabla\psi +
 \int_{\D\setminus\Omega}\nabla(\tilde{v}_2-\tilde{u}_2)\cdot\nabla\psi\right|\\
\endaligned
$$
which in turn is controlled, using \eqref{eq4}, by a multiple of
$$\aligned
\int_\D|\nabla(\tilde{v}_2-\tilde{u}_2)|\,|\nabla\psi|
&\leq\|\nabla(\tilde{v}_2-\tilde{u}_2)\|_{L^2(\D)}\,\|\nabla\psi\|_{L^2(\D)}\\
&\leq
c\,\rho\,\|\varphi\|_{H^\frac12(\partial\D)}\,\|\psi_0\|_{H^\frac12(\partial\D)}.
\endaligned$$
This gives for \eqref{eq5} that the difference of
Dirichlet-to-Neumann maps satisfies
$$
|\langle(\Lambda_{\tilde{\gamma}_1}-\Lambda_{\tilde{\gamma}_2})(\varphi_0),\psi_0\rangle|\leq
c\,\rho\,\|\varphi\|_{H^\frac12(\partial\D)}\,\|\psi_0\|_{H^\frac12(\partial\D)}
$$
as desired.
\end{proof}

\begin{Rem}
The trivial extension of the conductivities by $1$ simplifies the
arguments but has the prizes or loosing regularity if $\alpha\ge
1/2$. An argument similar to that in \cite{BFR} would need an
$L^2$ version of the boundary recovery result of Brown (see also
\cite{A1}) of the type
\[ \|\gamma_1-\gamma_2\|_{L^2(\partial \Omega)} \le C\rho \]
\end{Rem}
\noindent

\section{Beltrami equations and fractional Sobolev spaces}\label{beltramifractsob}

This section is devoted to investigate how quasiconformal mappings
interplay with fractional Sobolev spaces. We face three different
goals. First, given a Beltrami coefficient $\mu\in
W^{\alpha,2}_0(\C)$, we find $\beta\in(0,\alpha)$ such that for
any $K$-quasiconformal mapping $\phi$ the composition
$\mu\circ\phi$, which is another Beltrami coefficient with the
same ellipticity bound, belongs to $W^{\beta, 2}(\C)$. Secondly,
we obtain the optimal (at least when $\alpha\approx 1$),  Sobolev
regularity for the homeomorphic solutions to the equation
$$\overline\partial f=\mu\,\partial f+\nu\,\overline{\partial f}$$
under the assumptions of ellipticity and Sobolev regularity for the coefficients. Finally, we obtain bounds for the complex geometric optics solutions. \\
Many properties of planar quasiconformal mappings rely on two precise integral operators, the Cauchy transform,
\begin{equation}\label{Cauchy} {\cal C} \varphi(z)= \frac{-1}{\pi}\, \int
\frac{\varphi(w)}{(w-z)}\,dA(w).\end{equation}
and the Beurling transform,
\begin{equation}\label{Beurling}
T\varphi(z)=\frac{-1}{\pi}\,\lim_{\varepsilon\to 0}
\int_{|w-z|\geq\varepsilon}\frac{\varphi(w)}{(w-z)^2}\,dA(w).
\end{equation}
Their basic properties are well known and can be found in any reference concerning planar quasiconformal mappings, \cite{Al,AIM,AP}.

\subsection{Composition with quasiconformal mappings}

Let $\mu$ be a compactly supported Beltrami coefficient, satisfying
$$
|\mu|\leq \frac{K-1}{K+1}=\kappa\,\chi_{\D}.
$$
Further, assume that
$$\mu\in W^{\alpha, 2}(\C)\text{ and }\|\mu\|_{W^{\alpha,2}(\C)}\leq \Gamma_0$$
for some $\alpha>0$ and some $\Gamma_0>0$. Let $\phi:\C\to\C$ be a planar $K$-quasiconformal mapping. In this section,
we look for those $\beta>0$ such that $\mu\circ\phi\in W^{\beta, 2}(\C)$.\\

%
%
\noindent
We need to recall a local version of a lemma due to Fefferman and Stein, see \cite{Pe} and \cite[Proposition
2.24]{Duo}. The proof follows from Vitali covering Lemma, exactly as in \cite{Pe}. By $Mf$ we denote the Hardy-Littlewood maximal function,
$$
Mf(x)=\sup\frac{1}{|D|}\int_D f.
$$
where the supremum runs over all disks $D$ with $x\in D$, while $M_\Omega f$ denote its local version, that is,
$$
M_\Omega f(x)=\sup \frac{1}{|D|}\int_D f
$$ 
where the supremum is taken over all discs $D$ with $x\in D\subset\Omega$. 

\begin{Lemma}\label{selfadjmaximal}
Let $w\geq 0$ a locally integrable function. Then
\[
\int_{\Omega} |M_\Omega f|^p \omega dx \le \int_{\Omega} |f|^p M\omega .\]
\end{Lemma}

\noindent
We can now prove the main result of this section.

\begin{Prop}\label{composition}
Let $K\geq 1$. Let $\mu\in W^{\alpha, 2}(\C)$ for some $\alpha\in(0,1)$, and assume that $|\mu|\leq\frac{K-1}{K+1}\,\chi_{|\D}$. Let $\phi:\C\to\C$ be any $K$-quasiconformal mapping, conformal out of a compact set, and normalized so that $|\phi(z)-z|\to 0$ as $|z|\to\infty$. Then
$$\mu\circ\phi\in W^{\beta,2}(\C)$$
whenever $\beta<\frac{\alpha}{K}$. Moreover,
$$
\|\mu\circ\phi\|_{W^{\beta,2}(\C)}\leq C\,\|\mu\|_{W^{\alpha,2}(\C)}^\frac{1}{K},
$$
for some constant $C>0$ depending only on $\alpha, \beta$ and $K$.
\end{Prop}
\begin{proof}
It is clear that $\mu\circ\phi$ belongs to $L^2(\C)$, so since $W^{\alpha,2}$ agrees with the Besov space $B^{2,2}_\alpha$, it suffices to show the convergence of the integral
$$\int_\C\int_\C\frac{|\mu(\phi(z+w))-\mu(\phi(z))|^2}{|w|^{2+2\beta}}\,dA(z)\,dA(w)$$
for every $\beta<\frac{\alpha}{K}$. First of all, note that by Koebe's $\frac14$ Theorem we have the inclusions
$\phi(\D)\subset 4\D$ and $\phi^{-1}(4\D)\subset 16\D$. Thus, for large $w$ there is nothing to say since
$$\aligned
\int_{|w|>1}\int_\C&\frac{|\mu(\phi(z+w))-\mu(\phi(z))|^2}{|w|^{2+2\beta}}\,dA(z)\,dA(w)\\
&\leq
2\,\|\mu\|_{L^2(\C)}^2\,\int_{|w|>1}\frac{1}{|w|^{2+2\beta}}dA(w)=\frac{C\,\|\mu\|_{L^2(\C)}^2}{\beta}
\endaligned$$
for some universal constant $C>0$. Then we are left to bound the integral
$$
\int_{|w|\leq 1}\int_\C\frac{|\mu(\phi(z+w))-\mu(\phi(z))|^2}{|w|^{2+2\beta}}\,dA(z)\,dA(w).
$$
Notice that this integral is in fact the same as
\begin{equation}\label{eq22}
\int_{|w|\leq 1}\int_F\frac{|\mu(\phi(z+w))-\mu(\phi(z))|^2}{|w|^{2+2\beta}}\,dA(z)\,dA(w)
\end{equation}
where $F=\{z\in\C:d(z,\phi^{-1}(\D))\leq 1\}\subset 17\D$ by Koebe's Theorem again. \\
\\
Since $\mu\in W^{\alpha,2}(\C)$, then by interpolation we get that $\mu\in W^{\alpha\theta,\frac{2}{\theta}}(\C)$ for each $\theta\in(0,1)$ and $$\|\mu\|_{W^{\alpha\theta,\frac{2}{\theta}}(\C)}\leq C\,\|\mu\|_\infty^{1-\theta}\,\|\mu\|_{W^{\alpha,2}(\C)}^{\theta}.$$
The goal is to choose $\theta$ to obtain larger possible $\beta$. For this, we use condition \eqref{pointwise}. Indeed, for each $\lambda\in(0,\alpha\theta)$ there exists a function $g=g_\lambda\in L^{p_\lambda}(\C)$, $p_\lambda=\frac{2}{\theta-(\alpha\theta-\lambda)}$, such that
$$|\mu(\zeta)-\mu(\xi)|\leq|\zeta-\xi|^{\lambda}\,\left(g(\zeta)+g(\xi)\right)$$
at almost every $\zeta,\xi\in\C$. Furthermore,
$$
\aligned
\|g_\lambda\|_{L^{p_\lambda}(\C)}
&\leq C_\lambda\,\|\mu\|_{W^{\alpha\theta,\frac{2}{\theta}}(\C)}\\
&\leq C_\lambda\,\|\mu\|_\infty^{1-\theta}\,\|\mu\|_{W^{\alpha,2}(\C)}^{\theta}.
\endaligned$$
with $C_\lambda$ bounded as $\lambda\to\alpha\theta$. Hence, if $|w|\leq 1$ then
$$\aligned
\frac{|\mu(\phi(z+w))-\mu(\phi(z))|}{|w|^\lambda} &\leq
\left(\frac{|\phi(z+w)-\phi(z)|}{|w|}\right)^{\lambda}\,\left(g(\phi(z+w))+g(\phi(z))\right).
\endaligned$$
Now use quasiconformality and the reverse H\"older inequality for the jacobian (see \cite{AIS} for a more precise formulation) to get that
$$\aligned
\left(\frac{|\phi(z+w)-\phi(z)|}{|w|}\right)^\lambda&\leq
C_K\,\left(\frac{\diam \phi(D(z,|w|))}{\diam D(z,|w|)}\right)^\lambda\\
&\leq
C_K\,\left(\frac{1}{|D(z,|w|)|}\,\int_{D(z,|w|)}J(\zeta,\phi)\,dA(\zeta)\right)^\frac{\lambda}{2}\\
&\leq C_K\,\left(M_\Omega J_\lambda(z)\right)^\frac{1}{2}
\endaligned$$
where $\Omega=\left\{z\in\C: d(z,\phi^{-1}(\D))\leq 2\right\}$ and $M_\Omega J_\lambda(z)$ denotes the local Hardy-Littlewood maximal function $M_\Omega$ at the point $z$ of $J(\cdot,\phi)^\lambda$. Note also that $\Omega\subset 18\D$ by Koebe's Theorem. By symmetry, we could also write $M_\Omega J_\lambda(z+w)$ instead of $M_\Omega J_\lambda(z)$, so we end up getting
$$
\frac{|\mu(\phi(z+w))-\mu(\phi(z))|^2}{|w|^{2\lambda}}
\leq C\,
\left(M_\Omega J_\lambda(z+w)\,g(\phi(z+w))^2+M_\Omega J_\lambda(z)\,g(\phi(z))^2\right).
$$
Therefore the integral at \eqref{eq22} is bounded by a constant times
$$
\int_{|w|\leq
1}\int_F\frac{M_\Omega J_\lambda(z+w)\,g(\phi(z+w))^2+M_\Omega J_\lambda(z)\,g(\phi(z))^2}{|w|^{2+2\beta-2\lambda}}\,dA(z)\,dA(w).
$$
If we restrict ourselves to values of $\lambda$ within the interval $(\beta,\alpha\theta)$, then the integral above is bounded by
\begin{equation}\label{eq26}
\frac{C}{\lambda-\beta}\,\int_F M_\Omega J_\lambda(z)\,g(\phi(z))^2\,dA(z).
\end{equation}
To get bounds for this, we start by choosing parameters. Fix $\alpha, \beta$ and $K$ with $\beta<\frac{\alpha}{K}$. Then we can find $s>1$ such that $\beta s<\frac{\alpha}{K}$. Now let us consider real numbers $\theta\in(0,1)$ and $\lambda\in(\beta,\alpha\theta)$ satisfying
\begin{equation}\label{eq23}
\lambda+(1-\alpha)\theta<\frac{1}{Ks}.
\end{equation}
Such conditions are compatible precisely when $\beta s<\frac{\alpha}{K}$. Condition \eqref{eq23} also guarantees that
$p_\lambda>2Ks$, so we can find $r$ satisfying
\begin{equation}\label{eq24}
1+\lambda s(K-1)<r<\frac{p_\lambda}{2Ks}(1+\lambda s(K-1)),
\end{equation}
and obviously $r>1$.\\
Once the parameters have been chosen, we proceed as follows. First, by H\"older's inequality
$$\aligned
\int_F M_\Omega J_\lambda(z)\,g\circ\phi(z)^2\,dA(z)
&=\int_\Omega M_\Omega J_\lambda(z)\,g\circ\phi(z)^2\,\chi_F(z)\,dA(z)\\
&\leq \left(\int_\Omega M_\Omega J_\lambda(z)^s\,g\circ\phi(z)^{2s}\,\chi_F(z)\,dA(z)\right)^\frac1s\,|F|^{1-\frac1s}.
\endaligned$$
Now we use Lemma \ref{selfadjmaximal} to bound the last integral above by a constant times
$$
\int_\Omega J(z,\phi)^{\lambda s}\,M(g\circ\phi^{2s}\,\chi_F)(z)\,dA(z).
$$
For any $r>1$, we can bound the above integral by
$$
\left(\int_\C J(z,\phi)^{\lambda s}\,M(g\circ\phi^{2s}\,\chi_F)(z)^r\,dA(z)\right)^\frac1r\,\left(\int_\Omega J(z,\phi)^{\lambda s}dA(z)\right)^{1-\frac1r}.
$$
The first inequality at \eqref{eq24} guarantees that the weight $J(\cdot,\phi)^{\lambda s}$ belongs to the Muckenhoupt class $A_r$ (see \cite{AIS} for details). Therefore we can use the weighted $L^r$ inequality for the maximal function and a change of coordinates to see that
$$\aligned
\int_\C J(z,\phi)^{\lambda s}\,M(g\circ\phi^{2s}\,\chi_F)(z)^r\,&dA(z)
\leq C_{r}\int_F J(z,\phi)^{\lambda s}\,g\circ\phi(z)^{2sr}\,dA(z)\\
&=C_r\,\int_{\phi(F)}J(w,\phi^{-1})^{1-\lambda s}\,g(w)^{2sr}\,dA(w),
\endaligned
$$
where $C_r$ is a positive constant depending on $r$. Summarizing, we get for the integral at \eqref{eq26} the bound
$$
C\,|F|^{1-\frac{1}{s}}\,
\left(\int_\Omega J(z,\phi)^{\lambda s}dA(z)\right)^{\frac1s-\frac{1}{sr}}\,\left(\int_{\phi(F)} J(w,\phi^{-1})^{1-\lambda s}\,g(w)^{2sr}\,dA(w)\right)^{\frac{1}{rs}}\\
$$
For the second integral above, we only use H\"older's inequality again, and obtain the bound
$$
\left(\int_{\phi(F)}g(w)^{p_\lambda}dA(w)\right)^\frac{2rs}{p_\lambda}\,\left(\int_{\phi(F)} J(w,\phi^{-1})^{\frac{p_\lambda(1-\lambda s)}{p_\lambda-2rs}}dA(w)\right)^{\frac{p_\lambda-2rs}{p_\lambda}}
$$
which is finite provided that both $p_\lambda>2rs$ and $\frac{p_\lambda(1-\lambda s)}{p_\lambda-2rs}<\frac{K}{K-1}$ hold. But both facts are guaranteed by our choice of parameters, in particular to the second inequality at \eqref{eq24}. This means that the integral at \eqref{eq26} has the upper bound
$$
\frac{C\,|F|^{1-\frac{1}{s}}}{\lambda-\beta}\,
\left(\int_\Omega J(z,\phi)^{\lambda s}dA(z)\right)^{\frac1s-\frac{1}{sr}}\,\|g\|_{L^{p_\lambda}(\phi(F))}^2\,\left(\int_{\phi(F)} J(w,\phi^{-1})^{\frac{p_\lambda(1-\lambda s)}{p_\lambda-2rs}}dA(w)\right)^{\frac{p_\lambda-2rs}{p_\lambda}}
$$
Since both $\phi$ and $\phi^{-1}$ are normalized $K$-quasiconformal mappings, the two integrals above are bounded by constants depending only on $K$. One obtains for the integral at \eqref{eq22} the bound
$$
\frac{C}{(\lambda-\beta)^\frac12}\|g\|_{L^{p_\lambda}(\phi(F))}\leq \frac{C}{(\lambda-\beta)^\frac12}\,\|\mu\|^{1-\theta}_{L^\infty(\C)}\,\|\mu\|^\theta_{W^{\alpha,2}(\C)}
$$
where the constant $C$ depends on $r, s, \lambda, \theta, \alpha$ and $K$. \\
To find larger possible $\beta$, we have to find the supremum of those $\lambda$ for which the pair $(\theta,\lambda)$ belongs to the set
$$
A=\left\{(\theta,\lambda); 0<\theta<1,\,\,
\beta<\lambda<\theta\alpha,\,\,\lambda+(1-\alpha)\theta<\frac{1}{Ks}\right\}.
$$
according to \eqref{eq23}. This supremum is easily seen to be $\frac{\alpha}{Ks}$. Further, all the above argument works for every $s\in(1,\frac{\alpha}{\beta K})$, so that the bound for \eqref{eq22} reads now as
$$
\frac{C_K}{\frac{\alpha}{K}-\beta}\,\|\mu\|^\frac1K_{W^{\alpha,2}(\C)}.
$$
Summarizing,
$$\aligned
\int_\C\int_\C&\frac{|\mu(\phi(z+w))-\mu(\phi(z))|^2}{|w|^{2+2\beta}}dA(z)dA(w)\\
&\leq
C_K\left(\frac{\|\mu\|_{L^2(\C)}^2}{\beta}+\frac{1}{\frac{\alpha}{K}-\beta}\,\|\mu\|_{W^{\alpha,2}(\C)}^{\frac{2}{K}}\right)
\endaligned$$
for all $\beta\in(0,\frac{\alpha}{K})$. Equivalently, we have an
inequality for the nonhomogeneous norms
$$
\|\mu\circ\phi\|_{W^{\beta,2}(\C)}\leq
C(\alpha,\beta,K)\,\|\mu\|_{W^{\alpha,2}(\C)}^\frac{1}{K},
$$
as stated.
\end{proof}

\noindent

\begin{Rem}
The condition $\beta<\frac{\alpha}{K}$ is by no means sharp. This
is clear  when $\alpha$ is close to $1$ but also it can be seen
from the fact that we are using the H\"older regularity of $\phi$.
It seems that one can get a factor which runs between  $1/K$ and
$1$ in terms of $\alpha$. As promised in the introduction this
will be a matter of a forthcoming work.
\end{Rem}

\subsection{Regularity of homeomorphic solutions}

We start by recalling the basic result on the existence of
homeomorphic solutions to Beltrami type equations. In absence of
extra regularity the integrability of the solutions comes from the
work of Astala \cite{Ast}. We recall the proof in terms of Neumann series
since it will be used both in this section and in the sequel.

\begin{Lemma}\label{homeoexistence}
Let $\mu,\nu$ be bounded functions, compactly supported in $\D$, such that $||\mu(z)|+|\nu(z)||\leq\frac{K-1}{K+1}$ at almost every $z\in\C$. The equation
\begin{equation}
\overline\partial f=\mu\,\partial f+\nu\,\overline{\partial f}
\end{equation}
admits only one homeomorphic solution $\phi:\C\to\C$, such that $|\phi(z)-z|={\cal O}(1/|z|)$ as $|z|\to\infty$. Further, if $p\in(2,\frac{2K}{K-1})$ then the quantity $$\|\partial\phi-1\|_{L^p(\C)}+\|\overline\partial\phi\|_{L^p(\C)}$$
is bounded by a constant $C=C(K,p)$ that depends only on $K$ and $p$.
\end{Lemma}
\begin{proof}

 Put $\phi(z)=z+{\cal C}h(z)$, where $h$ is defined by
$$(I-\mu\,T-\nu\,\overline{T})h=\mu+\nu.$$
and ${\cal C}$ and $T$ denote, respectively, Cauchy and Beurling transforms.
Since $T$ is an isometry in $L^2(\C)$, one can
construct such a function $h$ as Neumann series
$$h=\sum_{n=0}^\infty(\mu T+\nu\overline{T})^n(\mu+\nu)$$
which obviously defines an $L^2(\C)$ function. By Riesz-Thorin
interpolation theorem,
$$\lim_{p\to 2}\|T\|_p=1,$$
it then follows that $h\in L^p(\C)$ for every $p>2$ such that $\|T\|_p<\frac{K+1}{K-1}$.
 Hence, the Cauchy transform ${\cal C}h$ is H\"older continuous (with exponent $1-\frac{2}{p}$). Further,
 since $h$ is compactly supported, we get $|\phi(z)-z|=|{\cal C}h(z)|\leq\frac{C}{|z|}$, and in fact
 $\phi-z$ belongs to $W^{1,p}(\C)$ for such values of $p$. A usual topological argument
  proves that $\phi$ is a homeomorphism.
  For the uniqueness, note that if we are given two solutions
  $\phi_1,\phi_2$ as in the statement then $\overline\partial(\phi_1\circ\phi_2^{-1})=0$ so
  that $\phi_1\circ\phi_2^{-1}(z)-z$ is holomorphic on $\C$ and vanishes at infinity.\\
In order to recover the  precise range
$(\frac{2K}{K+1},\frac{2K}{K-1})$ obtained by Astala \cite{Ast}
we need a remarkable result from \cite{AIS} which says that
$I-\mu\,T-\nu\,\overline{T}:L^p(\C)\to L^p(\C)$ defines a bounded
invertible operator for these values of $p$. Further, both its
$L^p$ norm and that of its inverse depend only on $K$ and $p$.
This implies that for every $p\in(\frac{2K}{K+1},\frac{2K}{K-1})$
there is a constant $C=C(K,p)$ such that
$$\|h\|_{L^p(\C)}\leq C_{K,p}.$$
The claim  follows since $\partial\phi-1=Th$ and
$\overline\partial\phi=h$.
\end{proof}
\noindent
Once we know about the existence of homeomorphic solutions, it is
time to check their regularity when the coefficients belong to
some fractional Sobolev space.

\begin{Theorem}\label{homeoregularity}
Let $\alpha\in (0,1)$, and suppose that $\mu,\nu\in W^{\alpha, 2}(\C)$ are Beltrami coefficients, compactly supported in $\D$, such that
$$\left||\mu(z)|+|\nu(z)|\right|\leq\,\frac{K-1}{K+1}.$$
at almost every $z\in\D$. Let $\phi:\C\to\C$ be the only homeomorphism satisfying
$$\overline\partial \phi=\mu\,\partial \phi+\nu\,\overline{\partial\phi}$$
and $\phi(z)-z={\cal O}(1/z)$ as $|z|\to\infty$. Then, $\phi(z)-z$ belongs to $W^{1+\theta\alpha, 2}(\C)$ for every $\theta\in(0,\frac{1}{K})$, and
$$
\|D^{1+\theta\alpha}(\phi-z)\|_{L^2(\C)}\leq C_K\,\left(\|\mu\|_{W^{\alpha,2}(\C)}^\theta+\|\nu\|_{W^{\alpha,2}(\C)}^\theta\right)
$$
for some constant $C_K$ depending only on $K$.
\end{Theorem}
\begin{proof}
We consider a ${\cal C}^\infty$ function $\psi$, compactly supported inside of $\D$, such that $0\leq\psi\leq 1$ and $\int\psi=1$. For $n=1,2,...$ let $\psi_n(z)=n^2\,\psi(nz)$. Put
$$\mu_n(z)=\int_\C\mu(w)\,\psi_n(z-w)\,dA(w),$$
and
$$\nu_n(z)=\int_\C\nu(w)\,\psi_n(z-w)\,dA(w).$$
It is clear that both $\mu_n$, $\nu_n$ are compactly supported in $\frac{n+1}{n}\D$, $|\mu_n(z)|+|\nu_n(z)|\leq\frac{K-1}{K+1}$, $\|\mu_n-\mu\|_{W^{\alpha,2}(\C)}\to 0$ and $\|\nu_n-\nu\|_{W^{\alpha,2}(\C)}\to 0$ as $n\to\infty$. Indeed there is convergence in $L^p$ for all $p\in(1,\infty)$. Thus, by interpolation we then get that for any $0<\theta<1$
$$\lim_{n\to\infty}\|\mu_n-\mu\|_{W^{\alpha\theta,\frac{2}{\theta}}(\C)}+\|\nu_n-\nu\|_{W^{\alpha\theta,\frac{2}{\theta}}(\C)}=0$$
and in particular, the sequences $D^{\alpha\theta}\mu_n$ and $D^{\alpha\theta}\nu_n$ are bounded in $L^\frac{2}{\theta}(\C)$.\\
Let $\phi_n$ be the only $K$-quasiconformal mapping $\phi_n:\C\to\C$ satisfying
\begin{equation}\label{qe}
\overline\partial\phi_n=\mu_n\,\partial\phi_n+\nu_n\overline{\partial\phi_n}
\end{equation}and normalized by $\phi_n(z)-z={\cal O}_n(1/z)$ as $|z|\to\infty$. By the construction in Lemma \ref{homeoexistence}, $\phi_n(z)=z+{\cal C}h_n(z)$ where $h_n$ is the only $L^2(\C)$ solution to
$$h_n=\mu_n\,Th_n + \nu_n\,\overline{T h_n} + (\mu_n+\nu_n),$$
and ${\cal C}h_n$ denotes the Cauchy transform. As in Lemma \ref{homeoexistence}, $h_n$ belongs to $L^p(\C)$ for all $p\in(\frac{2K}{K+1},\frac{2K}{K-1})$ and $\|h_n\|_{L^p(\C)}\leq C=C(K,p)$;
in particular, $\phi_n-z$ is a bounded sequence in $W^{1,p}(\C)$. \\
\\
We now write equation \eqref{qe} as
$$
\overline\partial (\phi_n-z)=\mu_n\partial(\phi_n-z)+\nu_n\overline{\partial(\phi_n-z)}+\mu_n+\nu_n
$$
and take fractional derivatives. If $\beta=\alpha\theta$, we can use Lemma \ref{KePV} (a) to find two functions $E_\beta$, $F_\beta$ such that
$$\aligned
D^{\beta}&\overline\partial (\phi_n -z)\\
&=D^{\beta}\mu_n\,\partial(\phi_n-z)+\mu_n\,D^{\beta}\partial(\phi_n-z)+E_{\beta}\\
&+D^{\beta}\nu_n\,\overline{\partial(\phi_n-z)}+\nu_n\,D^{\beta}\overline{\partial(\phi_n-z)}+F_{\beta}.
\endaligned$$
Further, $E_{\beta}$ satisfies
\begin{equation}\label{reg}
\|E_{\beta}\|_{L^2(\C)}\leq C_0\,\|D^{\beta}\mu\|_{L^{p_1}(\C)}\,\|\partial(\phi_n-z)\|_{L^{p_2}(\C)},
\end{equation}
where $p_2$ is any real number with $2<p_2<\frac{2K}{K-1}$ and $\frac{1}{p_1}+\frac{1}{p_2}=\frac12$, and $C_0$ depends on $p_1,p_2$. Analogously,
\begin{equation}\label{reg2}
\|F_{\beta}\|_{L^2(\C)}\leq C_0\,\|D^{\beta}\nu\|_{L^{p_1}(\C)}\,\|\partial(\phi_n-z)\|_{L^{p_2}(\C)}.
\end{equation}
Now we notice that   we  have $D^\beta\partial\varphi=\partial D^\beta\varphi$ and similarly for $\overline\partial$. Further, if $\varphi$ is real then $D^\beta \varphi$ is also real. Thus
$$\aligned
\overline\partial &D^\beta(\phi_n -z)\\
&=\mu_n\,\partial D^\beta(\phi_n-z)+D^\beta\mu_n\,\partial(\phi_n-z)+E_\beta\\
&+\nu_n\,\overline{\partial(D^{\beta}(\phi_n-z))}+D^{\beta}\nu_n\,\overline{\partial(\phi_n-z)}+F_{\beta},
\endaligned$$
or equivalently
$$\aligned
(I-\mu_n T-\nu_n \overline{T})&(\overline\partial)\left(D^\beta(\phi_n-z)\right)\\
&=D^\beta\mu_n\,\partial(\phi_n-z)+D^\beta\nu_n\,\overline{\partial(\phi_n-z)}+E_\beta+F_\beta.
\endaligned$$
The term on the right hand side is actually an $L^2(\C)$ function. To see this, it suffices to choose in both \eqref{reg} and \eqref{reg2} the value $p_1=\frac{2}{\theta}$ for some $\theta\in(0,\frac{1}{K})$. Now, the operator $I-\mu_n T-\nu_n\overline{T}$ is continuously invertible in $L^2(\C)$, and a Neumann series argument shows that the norm of its inverse is bounded by $\frac{1}{2}(K+1)$. Thus,
$$\aligned
\|\overline\partial &D^\beta(\phi_n-z)\|_{L^2(\C)}\\
&\leq C_0\,\frac{K+1}{2}\,\left(\|D^\beta\mu_n\|_{L^\frac{2}{\theta}(\C)}+\|D^\beta\nu_n\|_{L^\frac{2}{\theta}(\C)}\right)\,\|\partial(\phi_n-z)\|_{L^{p_2}(\C)}\\
&\leq C_0\,\frac{K+1}{2}\,\left(\frac{K-1}{K+1}\right)^{1-\theta}\left(\|\mu_n\|_{W^{\alpha,2}(\C)}^\theta+\|\nu_n\|_{W^{\alpha,2}(\C)}^\theta\right)\,\|\partial(\phi_n-z)\|_{L^{p_2}(\C)}\\
&\leq C_0\,\frac{K+1}{2}\left(\|\mu_n\|_{W^{\alpha,2}(\C)}^\theta+\|\nu_n\|_{W^{\alpha,2}(\C)}^\theta\right)\,\|\partial(\phi_n-z)\|_{L^{p_2}(\C)}
\endaligned$$
where $C_0$ is the constant in (\ref{reg}). As $n\to\infty$, the right hand side is bounded by
$$
C_0\,\frac{K+1}{2}\,\left(\|\mu\|_{W^{\alpha,2}(\C)}^\theta+\|\nu\|_{W^{\alpha,2}(\C)}^\theta\right)\,C(K,p_2)
$$
because $\|\partial(\phi-z)\|_{L^{p_2}(\C)}\leq C(K,p_2)$. Hence, $\overline\partial D^\beta(\phi_n-z)$ is bounded in $L^2(\C)$, and thus also $\partial D^\beta(\phi_n-z)$, because $T$ is an isometry of $L^2(\C)$ and $T(\overline\partial D^\beta(\phi_n-z))=\partial D^\beta(\phi_n-z)$. Therefore, by passing to a subsequence we see that $D^\beta(\phi_n-z)$ converges in $W^{1,2}(\C)$, and as a consequence $\phi-z$ belongs to $W^{1+\beta,2}(\C)$. Further, we have the bounds
$$\|D^{1+\theta\alpha}(\phi-z)\|_{L^2(\C)}\leq C\,\left(\|\mu\|_{W^{\alpha,2}(\C)}^\theta+\|\nu\|_{W^{\alpha,2}(\C)}^\theta\right)
$$
for some constant $C$ depending only on $K$.
\end{proof}

\subsection{Regularity of complex geometric optics solutions}

We are now ready to give precise bounds on the Sobolev regularity
of the complex geometric optics solutions to the equation
$\overline\partial f=\mu\,\overline{\partial f}$ introduced in
Theorem \ref{existuniqbeltrami}.

\begin{Theorem}\label{regcgos}
Let $\mu\in W^{\alpha, 2}(\C)$ be a Beltrami coefficient,
compactly supported in $\D$, with
$\|\mu\|_\infty\leq\frac{K-1}{K+1}$ and
$\|\mu\|_{W^{\alpha,2}(\C)}\leq\Gamma_0$. Let $f=f_\mu(z,k)$ the
complex geometric optics solutions to the equation
$$\overline\partial f=\mu\,\overline{\partial f}.$$
For any $0<\theta<\frac{1}{K}$ we have that
$$f\in W^{1+\theta\alpha,2}_{loc}(\C).$$
Further, we have the estimate
$$\|D^{1+\alpha\theta}(  f_\mu)(\cdot,k)\|_{L^2(\D)}\leq C(K)\,e^{C|k|}\,\left(1+|k|\right)\,\left(\Gamma_0+|k|^\alpha\right)^\theta$$
were $C, C(K)>0$, and $C(K)$ depends only on $K$.
\end{Theorem}
\begin{proof}
The existence and uniqueness of the complex geometric optics
solutions comes from \cite[Theorem 4.2]{AP} (see Theorem
\ref{existuniqbeltrami} in the present paper). Secondly, it is
shown in \cite[Lemma 7.1]{AP} that $f$ may be represented as
$$f=e^{ik\phi}$$
where $\phi:\C\to\C$ is the only $W^{1,2}_{loc}(\C)$ homeomorphism
solving
\begin{equation}\label{equationphi}
\overline\partial\phi=-\mu\,\frac{\overline{k}}{k}\,e_{-k}(\phi)\,\overline{\partial\phi}
\end{equation}
and normalized by the condition $|\phi(z)-z|\to 0$ as
$|z|\to\infty$.
Here $e_{-k}(\phi(z))=e^{-ik\phi(z)-i\overline{k}\overline{\phi(z)}}$ is a unimodular function
so that $|e_{-k}(\phi(z))|=1$. In
particular, $\phi$ is conformal outside of $\D$, because $\supp(\mu)\subset\D$. Thus by Koebe $\frac14$ theorem
\begin{equation}\label{linftyphi}
\phi(\D) \subset 4 \D  \Rightarrow  \|\phi (\cdot,k)\|_{L^\infty(\D)}
\le 4
\end{equation}
Our first task is to determine the Sobolev regularity of the coefficient $\mu\,e_{-k}(\phi)$ in equation \eqref{equationphi}. We will argue by interpolation. Firstly, by (\ref{linftyphi}),
\[\|e_{-k}(\phi)\|_{L^2(\D)}\le |\D|^\frac12.\]
For the $L^2$ norm of the derivative, we invoke Lemma \ref{homeoexistence} to obtain that
\[ \|D(\phi -z)\|_{L^2(\C)} \le C(K),\]
Thus,
$$
\|D(e_{-k}(\phi))\|_{L^2(\D)}\leq C(K)\,|k|\, |\D|^\frac12.
$$
and by interpolation we arrive to,
$$
\|e_{-k}(\phi)\|_{{W}^{\alpha,2}(\D)}\leq C(K)\,|\D|^\frac12\,|k|^\alpha.
$$
Now we will use the remark~\ref{LebnitzHolder} to see that that $\mu\,e_{-k}(\phi) $ belongs also to $W^{\alpha,2}(\C)$. Since $e_{-k}(\phi)$ is unimodular, the $L^2$ bound is obvious. By virtue of (\ref{KPVcompactsupport}) we have that
\begin{equation}\label{coefficient}
\begin{aligned}
\|D^\alpha(\mu\,e_{-k}(\phi))\|_{L^2(\C)}&\le
C(\|D^\alpha\mu\|_{L^2(\C)} \| e_{-k}(\phi)\|+ \kappa \|D^\alpha(e_{-k}(\phi))\|_{L^2(\D)}) \\
&\leq C\,|\D|^\frac12\,(|k|^\alpha+ \Gamma_0)\\
\end{aligned}
\end{equation}
The bound (\ref{coefficient}) allows us to apply Theorem \ref{homeoregularity} to equation \eqref{equationphi}. We obtain that  $\phi_0(z)=\phi(z)-z$ satisfies the estimate
\begin{equation}\label{phiregularity}
\|D^{\alpha\theta}(\partial\phi_0)\|_{L^2(\C)}+\|D^{\alpha\theta}(\overline\partial\phi_0)\|_{L^2(\C)}\leq
C_K\,(\Gamma_0+|k|^\alpha)^\theta
\end{equation}
for $\theta\in(0,\frac{1}{K})$. We push this bound to $f$. Since $f(z)=e^{ik\phi(z)}$, we have
$$\partial f(z)=e^{ik\phi(z)}\,ik\partial\phi(z)$$
and again from Lemma \ref{KePV} (c), for any disk $D$,
\begin{equation}\label{3terms}
\aligned \|D^{\alpha\theta}(\partial f)\|_{L^2(D)}\leq
&\|D^{\alpha\theta}(e^{ik\phi})\,ik\partial\phi\|_{L^2(D)}
+\|ik\,D^{\alpha\theta}(\partial\phi)\,e^{ik\phi}\|_{L^2(D)}\\
&+C\,\|e^{ik\phi}\|_{L^{\infty}(D)}\,\|D^{\alpha\theta}(\partial\phi)\|_{L^{2}(D)}.
\endaligned
\end{equation}
For the second and third terms on the right hand side above, we notice that \eqref{linftyphi} yields that
\begin{equation}\label{Linftybound}
\sup_{z\in \D}|e^{ik\phi(z)}|\leq e^{4|k|},
\end{equation}
which combined with \eqref{phiregularity} provides us with the estimate
\begin{equation}
\aligned
\|ik\,D^{\alpha\theta}(\partial\phi)\,e^{ik\phi}\|_{L^2(\D)}
&\leq |k|\,e^{C|k|}\,\|D^{\alpha\theta}(\partial\phi)\|_{L^2(\D)}\\
&\leq|k|\,e^{C|k|}\,C_K\,(\Gamma_0+|k|^\alpha)^\theta.
\endaligned
\end{equation}
Concerning the first term in \eqref{3terms}, we recall that $L^p$
bounds for $\partial\phi$ are possible only for
$p\in(\frac{2K}{K+1},\frac{2K}{K-1})$. This forces us to look for
$L^q$ bounds for $D^{\alpha\theta}(e^{ik\phi})$, for some $q>2K$.
These bounds are easily obtained by interpolation. More precisely,
we know the $L^\infty$ bound given at \eqref{linftyphi}. Further,
we have also a $W^{1,2}$ bound,
$$\aligned
\int_\D|\partial(e^{ik\phi(z)})|^2\,dA(z) \leq
|k|^2\,e^{C|k|}\,K\,\int_\D\,J(z,\phi)\,dA(z) \leq
|k|^2\,e^{C|k|}\,K\,|\D|.
\endaligned$$
Thus, by interpolation we obtain
$$\aligned
\|D^{\alpha\theta}(e^{ik\phi})\|_{L^\frac{2}{\alpha\theta}(\D)}
&\leq\,C\,\|e^{ik\phi}\|_{L^\infty(\D)}^{1-\alpha\theta}\,\|\partial(e^{ik\phi})\|_{L^2(\D)}^{\alpha\theta}\\
&\leq
|k|^{\alpha\theta}\,e^{C|k|}\,(K|\D|)^\frac{\alpha\theta}{2}\leq
C(K)\,|k|^{\alpha\theta}\,e^{C|k|}.
\endaligned$$
Now recall that $0<\theta<\frac{1}{K}$ is fixed, and let
$p=\frac{2}{\alpha\theta}$. If we now consider any real number $s$
such that $\frac{2}{1-\alpha\theta}<s<\frac{2K}{K-1}$, we obtain
$$\aligned
\|D^{\alpha\theta}(e^{ik\phi})\,ik\partial\phi\|_{L^2(\D)}
&\leq|k|\,\|D^{\alpha\theta}(e^{ik\phi})\|_{L^\frac{2}{\alpha\theta}(\D)}\,\|\partial\phi\|_{L^\frac{2}{1-\alpha\theta}(\D)}\\
&\leq C(K)\,|k|^{1+\alpha\theta}\,e^{C|k|}\,\|\partial\phi\|_{L^s(\D)}\\
&= C(K)\,|k|^{1+\alpha\theta}\,e^{C|k|}
\endaligned
$$
because the normalization on $\phi$ forces uniform bounds for
$\|\partial\phi\|_{L^s(\D)}$ depending only on $K$. Summarizing,
\eqref{3terms} gives us the bound
$$\|D^{\alpha\theta}(\partial f)\|_{L^2(\D)}\leq C(K)\,e^{C|k|}\,\left(1+|k|\right)\,\left(\Gamma_0+|k|^\alpha\right)^\theta.$$
Similar calculations give the corresponding bound for
$D^{\alpha\theta}(\overline\partial f)$.
\end{proof}

\noindent
We will also need the following bounds in Section \ref{endofproof}.

\begin{Lemma}\label{inversecgos}
Let $\mu$ be a Beltrami coefficient, compactly supported in $\D$.
Assume that $\|\mu\|_\infty\leq \frac{K-1}{K+1}$ and
$\|\mu\|_{W^{\alpha,2}(\C)}\leq\Gamma_0$. Let $f=f_\mu(z,k)$
denote the complex geometric optics solutions to
$$\overline\partial f=\mu\,\overline{\partial f}.$$
Let $p<2/(K-1)$.  Then, for any disk $D$
$$\int_{D}\left|\frac{1}{\partial f(z)}\right|^p\leq C$$
where the constant $C$ depends on $\diam(D)$, $k$, $K$ and
$\Gamma_0$.
\end{Lemma}
\begin{proof}
The function $f$ can be represented as
$$f=e^{ik\phi}$$
so that
$$\frac{1}{\partial f}=\frac{1}{e^{ik\phi}}\,\frac{1}{ik\partial\phi}.$$
As we have seen in \eqref{Linftybound} in the proof of the above Theorem
 gives also lower local uniform bounds for
$e^{ik\phi}$. Thus, only $L^p$ bounds for $\partial\phi$ are
needed. But these bounds come from the fact that $\phi$ is a
normalized $K$-quasiconformal mapping, so that
$$\partial\phi\in L^p$$
\end{proof}


\section{Uniform subexponential decay}\label{assymptotics}

We investigate the decay property of complex geometric optic solutions to the equation
$$\overline\partial f_{\lambda}=\lambda\mu\,\overline{\partial f_{\lambda}},$$
where $\lambda\in\partial\D$ is a fixed complex parameter, and $\mu\in W^{\alpha,2}_0(\C)$ is a Beltrami coefficient compactly supported in $\D$. It turns out that $f_{\lambda}$ admits the representation
$$f_{\lambda}(z,k)=e^{ik\phi_{\lambda}(z,k)}$$
where $\phi_{\lambda}$ satisfies the following properties (see \cite[Lemma 7.1]{AP} or the proof of Theorem \ref{regcgos} above):
\begin{enumerate}
\item $\phi_{\lambda}(\cdot,k):\C\to\C$ is a quasiconformal mapping.
\item $\phi_{\lambda}(z,k)=z+{\cal O}_k(1/z)$ as $|z|\to\infty$
\item $\phi_{\lambda}$ satisfies the {\em{nonlinear}} equation
\begin{equation}\label{nonlinearbeltrami}
\overline\partial\phi_{\lambda}(z, k)=-\lambda\,\mu(z)\,\frac{\overline k}{k}\,e_{-k}(\phi_{\lambda}(z,k))\,\overline{\partial\phi_{\lambda}(z,k)}
\end{equation}
\end{enumerate}
As was explained in Section 2, our goal is to obtain a uniform decay of the type
\begin{equation}\label{decay2}
|\phi_\lambda(z,k)-z|\leq \frac{C}{|k|^{b\alpha}}
\end{equation}
The precise statement can be found at Theorem \ref{nonlineardecay}. For the proof, we will mainly follow the
lines of both \cite{AP,BFR}. This consists on investigating first the behaviour of {\em{linear}} Beltrami equations with the rapidly oscillating coefficients $\mu(z)\,e_{-k}(z)$, and then treat the nonlinearity as a perturbation.


\subsection{Estimates for the linear equation}\label{linearsubsection}
\noindent
As usually, $\mu$ denotes a Beltrami coefficient, compactly supported in $\D$, with the ellipticity bound
$$\|\mu\|_{L^\infty(\Omega)}\leq\frac{K-1}{K+1}=\kappa$$
and the smoothness assumption
$$\|\mu\|_{W^{\alpha,2}(\C)}=\|\mu\|_{L^2(\C)}+\|D^\alpha\mu\|_{L^2(\C)}\leq\Gamma_0$$
for some $0<\alpha<1$ and $\Gamma_0>0$. For each complex numbers
$k\in\C$ and $\lambda\in\D$, let $\psi=\psi_\lambda(z,k)$ be the
only homeomorphic solution to the problem,
\begin{equation}
\begin{cases}
\overline\partial\psi(z,k)=\frac{\overline{k}}{k}\,\lambda\,e_{-k}(z)\,\mu(z)\,{\partial\psi(z,k)}\\
\psi(z,k)-z={\cal O}(1/z),\hspace{1cm}z\to\infty
\end{cases}
\end{equation}
Then $\psi$ can be represented by means of a Cauchy transform
\begin{equation}
\psi(z,k)-z=\int_\C \overline \partial \psi(w,k) \Phi(z,w) dA(w),
\end{equation}
where $\Phi(z,w)=\frac{\psi_\D(w)}{z-w}$ for  a smooth cutoff function $\psi_\D=1$ on $\D$ (in particular on the support of $\overline\partial\psi$). We need subtle properties for both terms.

\begin{Lemma}\label{decomposition}
Let $n_0$ be given, and let $s\ge 2$ be such that
$$\kappa\,\|T\|_s<1.$$
There exists a decomposition
 $\overline\partial\psi_\lambda(z,k)=g_\lambda(z,k)+h_\lambda(z,k)$ satisfying the following properties:
\begin{enumerate}
\item $\|h_\lambda(\cdot, k)\|_s\leq C(\kappa,
s)\,\left(\kappa\,\|T\|_s\right)^{n_0}.$
\item $\|g_\lambda(\cdot,
k)\|_s\leq C(\kappa).$
\item For  $R>0$ and  $|k|>2R$,
$$\left(\int_{|\xi|<R}|\widehat{g}_\lambda(\xi,k)|^qdA(\xi)\right)^\frac{1}{q}\leq C(\alpha,\kappa,p)\,M(p)^{n_0}\,\frac{\Gamma_0}{|k|^{\alpha}}
$$
where $1<p<2$, $q=\frac{p}{p-1}$, $\widehat{g}_\lambda(\xi,k)=({g}_\lambda(\cdot,k))^\wedge(\xi)$ and the
value of $M(p)$ is given in (\ref{Mp}).
\end{enumerate}
\end{Lemma}
\noindent The proof will rely on the Neumann series expression of $\overline\partial\psi$. For this, we consider the unimodular factors
$$e_n(z)=e^{in(kz+i\overline{k}\overline{z})}.$$
An idea which goes back to \cite{AP} is dealing with the unimodular factors $e_n$ conjugating them with the Beurling transform. Namely, we express
\begin{equation}\label{Neumann}
\overline\partial\psi=\sum_n \left(\frac{\overline{k}}{k}\lambda\right)^{n+1}e_{-(n+1)}\, f_n
\end{equation}
where
\begin{equation}\label{neumannseries}
\begin{cases}
f_0=\mu\\
f_n=\mu\, T_n (f_{n-1}),\,n=1,2,\dots
\end{cases}
\end{equation}
Here by $T_n$ we denote a singular integral operator defined by the rule
$$T_n(\varphi)=e_n\,T(e_{-n}\varphi)$$
where
 $T$ is the usual Beurling transform \eqref{Beurling}. It is not hard to see that $T_n$ is represented, at the frequency side, by a unimodular multiplier of the form
$$\widehat{T_n\varphi}(\xi)=\frac{\xi-n}{\overline{\xi - n}}\widehat\varphi(\xi)$$
Thus,
$$\|T_n\|_{L^2(\C)}=\|T_n\|_{L^2(\C)\to L^2(\C)}=1$$
and $T_n$ is an isometry of $L^2(\C)$. In fact, for any $1<p<\infty$,
$$\|T_n(\varphi)\|_{L^p(\C)}=\|T(e_{-n}\varphi)\|_{L^p(\C)}\leq \|T\|_{L^p(\C)}\,\|\varphi\|_{L^p(\C)}$$
because $|e_n(z)|=1$, so that $\|T_n\|_{L^p(\C)}=\|T\|_{L^p(\C)}$.
As $T_n$ is given by a Fourier multiplier, it commutes with any constant coefficients differential operator $D$ and  thus,
$$
\|T_n\varphi\|_{W^{1,p}(\C)}=\|T_n\varphi\|_{L^p(\C)}+\|T_n(D\varphi)\|_{L^p(\C)}\leq\|T\|_p\,\|\varphi\|_{W^{1,p}(\C)}
$$
and therefore $\|T_n\|_{W^{1,p}(\C)}\leq\|T\|_{L^p(\C)}$. Furthermore, the complex interpolation method gives that for any $0<\beta<1$,
\begin{equation}\|T_n\|_{W^{\beta, p}(\C)}\leq C_0\,\|T\|_{L^p(\C)}
\end{equation}
where $C_0>0$ is a universal constant.\\
\\
Let $1<p<2$ be fixed. We declare
\begin{equation}\label{Mp} \aligned
B&=B(p)=\|T_n\|_{W^{\alpha,p}(\C)}\leq\|T\|_{L^p(\C)}\\
D&=D(p)=\|T_n\|_{L^\frac{2p}{2-p}(\C)}\leq\|T\|_{L^\frac{2p}{2-p}(\C)}\\
M&=M(p)=B+D
\endaligned
\end{equation}
It is well known that the series $\sum_n f_n$ defines a compactly supported $L^2(\C)$ function (actually $L^p(\C)$ for any $1+\kappa<p<1+\frac{1}{\kappa}$). Next lemma yields Sobolev estimates for $f_n$  in terms of the Sobolev norm $\|\mu\|_{W^{\alpha,2}(\C)}$.

\begin{Lemma}\label{neumannestimates}
For any $1<p<2$ there exists a constant  $C=C(p)$ such that
$$\|f_n\|_{W^{\alpha,p}(\C)}\leq C(p)\,\Gamma_0\,\kappa^{n-1}\,(M(p))^n,$$
for any $n=1,2,...$.
\end{Lemma}
\begin{proof}
To prove the Lemma, we will use Remark~\ref{LebnitzHolder} for the complex dilatation  $\mu$ and an arbitrary  function $g \in W^{\alpha,p}$. Then it holds that
\begin{equation}\label{Leibnitzmu}
\|D^\alpha(\mu g)\|_{L^p(\C)} \leq \kappa \|D^\alpha g\|_{L^p(\D)}+C\,\Gamma_0 \|g\|_{L^{\frac{2p}{2-p}}(\C)}
\end{equation}
for some positive constant $C=C(p)\geq 1$. First of all, we study the $L^p$ norm of $f_n$. Recalling that $\mu$ is compactly supported inside of $\D$, we first see that
$$
\|f_n\|_{L^p(\C)}=\|f_n\|_{L^p(\D)}\leq \kappa\,\|T_nf_{n-1}\|_{L^p(\D)}.
$$
Next, (\ref{Leibnitzmu}) yields that,
$$\aligned
\|D^{\alpha}f_n\|_{L^p(\C)}&=\|D^\alpha(\mu\,T_n f_{n-1})\|_{L^p(\C)}\\
&\leq
C\,\Gamma_0\,\|T_nf_{n-1}\|_{L^\frac{2p}{2-p}(\C)}+\kappa\,\|D^{\alpha}T_nf_{n-1}\|_{L^p(\D)}
\endaligned$$
Hence, for any $n>1$,
$$\aligned
\|f_n\|_{W^{\alpha,p}(\C)} &= \|f_n\|_{L^p(\C)}+\|D^{\alpha}f_n\|_{L^p(\C)}\\
&\leq
  C
  \,\Gamma_0\,\|T_nf_{n-1}\|_{L^\frac{2p}{2-p}(\C)}+\kappa
  \|T_nf_{n-1}\|_{W^{\alpha,p}(\C)}
\endaligned$$
To control the first term above, we see that
$$\aligned
\|T_nf_{n-1}\|_{L^\frac{2p}{2-p}(\C)}
&\leq D\,\|f_{n-1}\|_{L^\frac{2p}{2-p}(\C)}\leq (D\kappa)\,\|T_{n-1}f_{n-2}\|_{L^\frac{2p}{2-p}(\C)}\\
&\leq (D\kappa)^{n-1}\,\|T_{1}f_{0}\|_{L^\frac{2p}{2-p}(\C)}\leq (D\kappa)^{n}\,|\D|^{\frac{1}{p}-\frac{1}{2}}
\endaligned$$
and for the second , if $n>1$
$$
 \|T_nf_{n-1}\|_{W^{\alpha,p}(\C)}
\leq B\,\|f_{n-1}\|_{W^{\alpha,p}(\C)}. $$ If we denote
$X_n=\|f_n\|_{W^{\alpha, p}(\C)}$ then we have just seen that
\begin{equation}\label{Xn}
X_n\leq C_1\,(\kappa\,D)^n+(\kappa\,B)\,X_{n-1}
\end{equation}
whenever $n>1$, and where
$C_1=C\,\Gamma_0\,|\D|^{\frac{1}{p}-\frac12}$. For $n=1$ we
proceed differently. Since both $T_1f_0$ and $D^\alpha T_1f_0$
belong to $L^2(\C)$, we can use H\"older's inequality to get
$$\aligned
\|T_1f_0\|_{L^p(\D)}+\|D^\alpha T_1f_0\|_{L^p(\D)}
&\leq\left(\|T_1f_0\|_{L^2(\C)}+\|D^\alpha T_1f_0\|_{L^2(\C)}\right)\,|\D|^{\frac1p-\frac12}\\
&=\|T_1f_0\|_{W^{\alpha,2}(\C)}\,|\D|^{\frac1p-\frac12}\\
&\leq |\D|^{\frac1p-\frac12}\,B\,\|f_0\|_{W^{\alpha,2}(\C)}=|\D|^{\frac1p-\frac12}\,B\,\Gamma_0.
\endaligned$$
Thus
$$
X_1\leq C_1\,\kappa D+B\,|\D|^{\frac1p-\frac12}\,\Gamma_0
$$
Thus, after recursively using (\ref{Xn}), we end up with
$$
X_n\leq C_1\,\kappa^n\,\sum_{j=0}^{n-1}B^jD^{n-j}+(\kappa B)^n\,\frac{\Gamma_0}{\kappa}\,|\D|^{\frac1p-\frac12}\leq\tilde{C}_1\,\kappa^{n-1}\,(B\kappa+D)^n
$$
where
$\tilde{C_1}=\max\{C_1, {\Gamma_0} \,|\D|^{\frac1p-\frac12}\}$.
Note finally that
$$\tilde{C}_1\leq C(p)\,{\Gamma_0},$$ which yields
the claim.
\end{proof}
\noindent
In particular, every function $f_n$ of the Neumann series is compactly supported and belongs  to $L^p(\C)$ for any $p\in(1,\infty)$, and also to $W^{\alpha, p}(\C)$ for any $p<2$.
%

\begin{Lemma}\label{fourierdecay}
If $h$ belongs to $W^{\alpha, p}(\C)$ for some $1<p<2$, then
$$
\left(\int_{|\xi|>R}|\widehat{h}(\xi)|^q\,dA(\xi)\right)^\frac{1}{q}\leq\,C(p)\,\frac{\|h\|_{W^{\alpha,p}(\C)}}{R^{\alpha}}
$$
\end{Lemma}
\begin{proof}
We wil use the characterization in terms of Bessel potentials of $
W^{\alpha, p}(\C)$. Since the Fourier transform maps continuously
$L^p(\C)$ into $L^q(\C)$, we get that
$$
\left(\int_\C\left((1+|\xi|^2)^\frac{\alpha}{2}\,|\widehat{h}(\xi)|\right)^q\,dA(\xi)\right)^\frac{1}{q}\leq C(p)\,\|h\|_{\alpha, p}
$$
Thus, a simple computation yields
$$\aligned
\left(\int_{|\xi|>R}|\widehat{h}(\xi)|^q\,dA(\xi)\right)^\frac{1}{q}
&\leq
\left(\int_{|\xi|>R}\left(\frac{(1+|\xi|^2)^\frac{\alpha}{2}}{|\xi|^{\alpha}}\right)^q|\widehat{h}(\xi)|^qdA(\xi)\right)^\frac{1}{q}\\
&\leq\frac{1}{R^\alpha}\left(\int_{\C}(1+|\xi|^2)^\frac{\alpha q}{2}\,|\widehat{h}(\xi)|^qdA(\xi)\right)^\frac{1}{q}\\
&\leq C(p)\,\frac{\|h\|_{\alpha, p}}{R^\alpha}
\endaligned$$
and the result follows.
\end{proof}

\begin{proof}[Proof of Lemma \ref{decomposition}]

We use the Neumann series
\begin{equation}\label{Neumann2}
 \overline\partial\psi=\sum_n \left(\frac{\overline{k}}{k}\lambda\right)^{n+1}e_{-(n+1)k}\, f_n
\end{equation}
introduced before. Then, take $g=\sum_{n=0}^{n_0} \left(\frac{\overline{k}}{k}\,\lambda\,e_{-k}\,\mu
T\right)^n\left(\frac{\overline{k}}{k}\,\lambda\,e_{-k}\,\mu\right)$
and $h=\partial_{\overline{z}}\psi-g$. In this way, properties
$\emph{1}$ and $\emph{2}$ follow easily from the general theory of
the Beltrami equation, since
$$\aligned
\Big\Vert
\left(\frac{\overline{k}}{k}\,\lambda\,e_{-k}\,\mu T\right)^n
&\left(\frac{\overline{k}}{k}\,\lambda\,e_{-k}\,\mu\right)
\Big\Vert_s\\
&\leq\kappa\,\|T\|_s\,
\Big\Vert
\left(\frac{\overline{k}}{k}\,\lambda\,e_{-k}\,\mu T\right)^{n-1}
\left(\frac{\overline{k}}{k}\,\lambda\,e_{-k}\,\mu\right)
\Big\Vert_s\\
&\leq
(\kappa\,\|T\|_s)^n\,\|\mu\|_s=(\kappa\,\|T\|_s)^n\,\kappa\,|\D|^\frac{1}{s}.
\endaligned$$
For the proof of $\emph{3}$, we must use the regularity of $\mu$.
Use \ref{Neumann2} to write $g(z,k)=\sum_{n=0}^{n_0} G_n(k,z)$
where $G_n(z,k)=
\left(\frac{\overline{k}}{k}\lambda\right)^{n+1}e_{-(n+1)k}\,
f_n$. Then, Lemma~\ref{neumannestimates} can be applied to $f_n$.
The Fourier transform of $G_n(z,k)$ (with respect to the $z$
variable) reads as
$$
\widehat{G_n}(\xi,k)=\left(\frac{\overline{k}}{k}\lambda\right)^{n+1}\,\widehat{f_n}(\xi-(n+1)k)
$$
Hence,  for $|k|>R$,  we can use lemma  \ref{fourierdecay}, to get
$$\aligned
\left(\int_{|\xi|<R}|\widehat{g}(\xi,k)|^q\,dA(\xi)\right)^\frac{1}{q}
&\leq\sum_{n=0}^{n_0}\left(\int_{|\xi|<R}|\widehat{G_n}(\xi,k)|^q\,dA(\xi)\right)^\frac{1}{q}\\
&=\sum_{n=0}^{n_0}\left(\int_{|\xi|<R}|\widehat{f_n}(\xi-(n+1)k)|^q\,dA(\xi)\right)^\frac{1}{q}\\
&=\sum_{n=0}^{n_0}\left(\int_{|\zeta+(n+1)k|<R}|\widehat{f_n}(\zeta)|^q\,dA(\zeta)\right)^\frac{1}{q}\\
&\leq\sum_{n=0}^{n_0}\left(\int_{|\zeta|>(n+1)|k|-R}|\widehat{f_n}(\zeta)|^q\,dA(\zeta)\right)^\frac{1}{q}\\
&\leq C(p)\,\sum_{n=0}^{n_0}\frac{\|f_n\|_{\alpha,p}}{((n+1)|k|-R)^{\alpha}}
\endaligned$$
where $C(p)$ is the constant from Lemma \ref{fourierdecay}. Now, using Lemma \ref{neumannestimates},
$$\aligned
\left(\int_{|\xi|<R}|\widehat{g}(\xi,k)|^q\,dA(\xi)\right)^\frac{1}{q}
&\leq
C(\kappa,p)\,\frac{\Gamma_0}{\kappa}\,\sum_{n=0}^{n_0}\frac{(\kappa\,M(p))^n}{((n+1)|k|-R)^{\alpha}}\\
&\leq
C(\kappa,p)\,(\kappa\,M(p))^{n_0}\,\frac{\Gamma_0}{\kappa}\,\sum_{n=0}^{n_0}\frac{1}{((n+1)|k|-R)^{\alpha}}
\endaligned
$$
and if we take $|k|\geq 2R$, then we finally get
$$
\aligned
\left(\int_{|\xi|<R}|\widehat{g}(\xi,k)|^q\,dA(\xi)\right)^\frac{1}{q}
&\leq
C(\alpha,\kappa,p)\,(\kappa\,M(p))^{n_0}\,\frac{\Gamma_0}{\kappa}\,
\frac{1}{|k|^{\alpha}}\,\sum_{n=0}^{n_0}\frac{1}{(n+\frac{1}{2})^{\alpha}}\\
&\leq
C(\alpha,\kappa,p)\,M(p)^{n_0}\,\frac{\Gamma_0}{\kappa}\,\frac{1}{|k|^{\alpha}}\\
&\leq C(\alpha,\kappa,p)\,M(p)^{n_0}\,\frac{\Gamma_0}{|k|^{\alpha}}
\endaligned
$$
and the result follows.
\end{proof}

\noindent The Cauchy kernel is not in $L^2$ but it belongs locally to $W^{\epsilon,p}$ for $1<p<2$, $\epsilon<\frac{2-p}{p}$. Thus we can work with a mollification of it which is perfectly controlled. However we need to choose carefully the mollification kernel (see \cite{Tre} vol 1 \&V.1).

\begin{Lemma}
There exists a $C_* >0$  such that  for  any $N>0$, there exists a $C^\infty$ function $\phi_N$ in $\C$ having  the following properties:
\begin{itemize}
\item $0\leq \phi_N\leq 1$, $ \phi_N= 1$ on  $ \D$ and $\phi_N= 0$ on $2\D$. \item $\int \phi_N=1$.
\item $ |D^\alpha \phi_N| \leq (C_*N)^{|\alpha |} $ for any $\alpha\in {\mathbb Z}^2_+$ with $|\alpha |\leq N$.
\end{itemize}
\end{Lemma}

\begin{Lemma}\label{nucli}
Let $\Phi(z,w)=\frac{\psi_\D}{z-w}$ and  $1<p<2$.
\begin{itemize}
\item[(a)] $\|\Phi(\cdot,z)\|_{L^p(\D)}\leq C(p)$ for all $z\in\C$.
\item[(b)] $\Phi(\cdot,z) \in  W^{\epsilon,p}$ for $\epsilon <\frac{2-p}{p}$ uniformly in $z$.
\end{itemize}
\begin{itemize}
\item[(c)] For any $N>0$, there exists a mollification $\Phi_{\delta,N}$ such that
$$\|\Phi(\cdot, z)-\Phi_{\delta,N}(\cdot, z)\|_{L^p(\D)}\leq C(\epsilon,p)\,\delta^\epsilon$$
whenever $z\in\C$ and $\epsilon<\frac{2-p}{p}$.
\item[(d)] $\|\Phi_{\delta,N}\|_{L^2(\C)}$ blows up as a power of $\delta$, i.e.
$$\|\Phi_{\delta,N}(\cdot, z)\|_{L^2(\C)}\leq C(p)\,\delta^{1-\frac{2}{p}}$$
\item[(e)] For each $R>\frac{1}{\delta}$ and $m>0$, there exists a universal constant $C_*$  and $C=C(p)$ such that for  any $m\leq N$
$$\|\widehat{\Phi_{\delta,N}}(\cdot,z)\|_{L^2(|\xi|\geq R)}\leq C(p)(C_*N)^m\,\delta^{1-\frac{2}{p}}\,(\delta R)^{-m}$$
\end{itemize}
\end{Lemma}
\begin{proof}
Claims (a) and (b) follow by the compactness of the support and Lemma \ref{KePV}. Now define
$$\widehat{\Phi_{\delta,N}(z, \cdot)}(\xi) =\widehat{\phi_N}(\delta\xi)\widehat{\Phi (z, \cdot)}(\xi).$$
Claim (c) follows  from the fact that since $p<2$, $W^{\epsilon,p}
\subset B_{\epsilon}^{p,2}$ (\ref{BesovSobolev}). Namely,
\[
\| \Phi_z(\cdot)-\Phi_{\delta,N}(z,\cdot)\|_{L^p} \le \int _\C\omega_p(\Phi_z)(w)
\phi_\delta(w)dw\]
\[\le  \|\Phi_z\|_{B_{\epsilon,p,2}} \int
(\phi_\delta(w))^2|w|^{2+\epsilon 2})^{\frac{1}{2}} \le
\delta^\epsilon (\int \phi^2(y) |y|^{2+\epsilon 2})^\frac{1}{2}
\le \delta^\epsilon \|\phi\|_{L^2(\C)}
\]
For claim (d), using Plancherel, H\"older, Hausdorff-Young inequalities and (a), we obtain, for $ 1/p-1/q=1/2$, that
\[ \|\Phi_{\delta,N}\|  _{L^2} \le \|\Phi_z\|_{L^p} \||\widehat{\phi_N}(\delta
\cdot) \|_{L^q}\le C \delta^{1-\frac{ 2}{p}}.\]
For the last claim, write again
 \[
 \begin{aligned}
 \|\widehat{\Phi_{\delta,N}}\|  _{L^2(|\xi|>R_0)} &\le \|\Phi_z\|_{L^p} \||\widehat{\phi_N}(\delta
\xi) \|_{L^q(|\xi|>R_0)}\\
&\le\|\Phi_z\|_{L^p}\delta^{1-2/p} \|\widehat{\phi_N}(
\xi) \|_{L^q(|\xi|>\delta R_0)}
 \end{aligned}
\]
Now
$$
\begin{aligned}
& \|\widehat{\phi_N}(\xi) \|_{L^q(|\xi|>\delta R_0)} \le \||\frac{(\xi_1+i\xi_2)^m}{|\xi|^m}|\widehat{\phi_N}(
\xi) \|_{L^q(|\xi|>\delta R_0)} \\
&\le (\delta R_0) ^{-m} \|\sum _{|\alpha|=m}\frac{m!}{\alpha!}|\widehat{D^{\alpha}\phi_N}(\xi)|� \|_{L^q}
\le (\delta R_0) ^{-m}\sum _{|\alpha|=m}\frac{m!}{\alpha!}\|  D^{\alpha}\phi_N
  \|_{L^{q'}}\\
 &\le (\delta R_0) ^{-m}\sum _{|\alpha|=m}\frac{m!}{\alpha!}(C_*N)^{m} \le(\delta R_0) ^{-m}(2C_*N)^{m}
 \end{aligned}
$$
for $m\le N$ from where (d) follows.
\end{proof}
\noindent
Now we combine the above estimates to obtain the precise decay for the solutions to the linear equation.

\begin{Prop}\label{lineardecay}
Assume that $\mu\in W^{\alpha,2}(\C)$ is a Beltrami coefficient, with compact support inside of $\D$, such that
$\|\mu\|_\infty\leq\kappa$ and $\|\mu\|_{W^{\alpha,2}(\C)}\leq\Gamma_0$. For each $\lambda\in\partial\D$ and each $k\in\C$, let $\psi=\psi_\lambda(z,k)$ be the quasiconformal mapping satisfying
\begin{equation}\label{almostnonlinearbeltrami}
\overline\partial\psi_\lambda(z,k)=\frac{\overline{k}}{k}\,\lambda\,e_{-k}(z)\,\mu(z)\,{\partial\psi_\lambda(z,k)}
\end{equation}
and normalized by
$$
\psi_\lambda(z,k)-z={\cal O}(1/z),\hspace{1cm}z\to\infty.
$$
There exists positive constants $C=C(\kappa)$ and $b=b(\kappa)$ such that
$$|\psi_\lambda(z,k)-z|\leq \frac{C \, \Gamma_0}{|k|^{b\,\alpha}}$$
for every $z,k\in\C$ and every $\lambda\in\partial\D$.
\end{Prop}
\begin{proof}
Let $b>0$ a constant to be defined , and let $n_0\in\N$. As in \cite{BFR}, we can represent
$$\aligned
\psi_\lambda(z,k)-z
&=C\,\int_\D\frac{\overline\partial\psi_\lambda(w,k)}{w-z}dA(w) \\
&=C\,\int_\C \Phi(w,z)\,(g(w,k)+h(w,k))dA(w)
\endaligned$$
with $g=g_\lambda(z,k)$ and $h=h_\lambda(z,k)$ as in Lemma \ref{decomposition}.\\
Recall that we have control on $\widehat{g}$ for low frequences by
property 3 in  Lemma~\ref{decomposition}, whereas $h$ will be
controlled by the ellipticity. It is also convenient to consider
the mollification $\Phi_{\delta, N}$ of $\Phi$ given in lemma
\ref{nucli}  for $N$ to be chosen along the proof.
 We will therefore estimate the following four terms separately.
 The first three are dealt with by the usual ellipticity theory and
 the Sobolev regularity of the Cauchy Kernel. Hence the estimates
 will depend on a suitable exponent $s=s(\kappa)$. It is in the
 last term where $\alpha,p$ will appear. Then we will chose the exponent $p=4/3$ which will yield better
 constants.

\begin{description}
\item[I=] $\displaystyle\int_\D \Phi(w,z)\,h(w)\,dA(w),$
\item[II=] $\displaystyle\int_\D(\Phi(w,z)-\Phi_{\delta,N}(w,z)) g \,dA(w),$ 
\item[III=] $\displaystyle\int_{|\xi|<R}\widehat{\Phi_{\delta,N}}(\xi,z)\,\widehat{g}(\xi,k)\,dA(\xi),$ 
\item[IV=] $\displaystyle\int_{|\xi|>R}\widehat{\Phi_{\delta,N}}(\xi,z)\,\widehat{g}(\xi,k)\,dA(\xi)$

\end{description}

\paragraph{$\mathbf{I} $:The tail} Fix $s=s(\kappa)$ such that $\kappa\,\|T\|_s<1$. Then we have
$$\aligned
\left|\int_\D \Phi(w,z)\,h(w)\,dA(w)\right|
&\leq\|\Phi(\cdot,z)\|_{L^\frac{s}{s-1}(\D)}\,\|h\|_{L^s(\D)}\\
&\leq C(\kappa, s)\,\left(\kappa\,\|T\|_s\right)^{n_0}
\endaligned$$
since by Lemma \ref{nucli} (a), the norm
$\|\Phi(\cdot,z)\|_{L^\frac{s}{s-1}(\D)}$ does not depend on $z$.
Take now,
\begin{equation}\begin{aligned}\label{no}
n_0 &\ge C(\kappa,s)+ b\,\frac{\log(|k|)}{-\log(\kappa\,\|T\|_s)}=C(\kappa)(1+b {\log(|k|})
\end{aligned}
\end{equation}
so that,
\begin{equation}\label{eleccio1}
C(\kappa, s)\,\left(\kappa\,\|T\|_s\right)^{n_0}\leq |k|^{-b}
\end{equation}
and hence
$$|\mathbf{I}|\leq |k|^{-b}.$$

\paragraph{$\mathbf{II}$: The error of mollification.} We will use Lemma \ref{nucli} with exponent $1<s'<2$ with $\frac{1}{s}+\frac{1}{s'}=1$. Thus $0<\epsilon<1-\frac{2}{s}$ and $0<\delta<|k|^{\frac{-b}{\epsilon}} $. Then it follows from Lemma \ref{nucli} (c) and Lemma \ref{decomposition} that
$$\aligned
|\mathbf{II}| \le
\|g\|_{L^s(\D)}\|\Phi(\cdot,z)-\Phi_{\delta,N}(\cdot,z)\|_{L^\frac{s}{s-1}(
\D)} \le C(\kappa,s,\epsilon)\,\delta^{\epsilon}\le  |k|^{-b}.
\endaligned$$

\paragraph{$\mathbf{III}$: The mollification  at high frequencies}. We now choose the optimal value of $N$.  We use Plancherel's Theorem and Lemma \ref{nucli} (e) with $m=N$ (assuming $R\delta>1$), to get
\begin{equation}
\aligned \left|\int_{|\xi|\geq R} \widehat{\Phi_{\delta,N}
}(\xi,z)\,\widehat{g}(\xi,k)\,dA(\xi)\right|
&\leq \|g\|_{L^2(\C)}\,\|\widehat{\Phi_{\delta,N}}(\cdot,z)\|_{L^2(|\xi|\geq R)}\\
&\leq C(s,\kappa) (C_*N)^N\,\delta^{\frac{2}{s}-1}\,(\delta R)^{-N} \le
|k|^{-b}.
\endaligned,
\end{equation}
Let us plug in the value of $\delta$ and choose $N$ to obtain the optimal value of $R$. Namely first $\delta^{\frac{2}{s}-1}\approx |k|^{2 b}$. Thus we obtain that
$$R^N \ge (C N)^N |k|^{3 b+\frac{N b}{\epsilon}}$$
or
$$R\ \ge C |k|^{\frac{b}{\epsilon}} N |k|^{\frac{3b}{N}}.$$
With the optimal $N=[3b\log (k)]+1]$ we get the condition
$$R \ge C |k|^{\frac{b}{\epsilon}}\log (|k|).$$
Imposing $b<\epsilon$ we obtain that for large $|k|$ it is enough to take
\begin{equation}
R \ge \frac{|k|}{4}.
\end{equation}
\paragraph{$\mathbf{IV}$: The mollification at  low frequencies.} The final term is the crucial one. Take $1<p<2$, and $q=\frac{p}{p-1}$. Then
$$\left|\int_{|\xi|<R}\widehat{g}(\xi,k)\,\widehat{\Phi_{\delta,N}}(\xi,z)\,dA(w)\right|\leq
\left(\int_{|\xi|<R}|\widehat{g}(\xi,k)|^q\,dA(\xi)\right)^\frac{1}{q}\,\|\widehat{\Phi_{\delta,N}}(\cdot,z)\|_{L^p(|\xi|<R)}$$
For $|k|\geq 2R$ we can use Lemma \ref{decomposition} and obtain
$$
\left(\int_{|\xi|<R}|\widehat{g}(\xi,k)|^q\,dA(\xi)\right)^\frac{1}{q}\leq
C(\alpha,\kappa,p)\,M(p)^{n_0}\,\frac{\Gamma_0}{|k|^{\alpha}}$$
At the same time, the other factor is bounded with the help of Lemma \ref{nucli} (d), which is allowed since $p<2$. More precisely, we have
$$\aligned
\|\widehat{\Phi_{\delta,N}}(\cdot,z)\|_{L^p(|\xi|<R)}
&=\left(\int_{|\xi|<R}|\widehat{\Phi_{\delta,N}}(\xi,z)|^p\,dA(\xi)\right)^\frac{1}{p}\\
&\leq C(p)\,R^{\frac{2}{p}-1}\,\left(\int_{|\xi|<R}|\widehat{\Phi_{\delta,N}}(\xi,z)|^2\,dA(\xi)\right)^\frac{1}{2}\\
&\leq C(p)\,R^{\frac{2}{p}-1}\,\|\Phi_{\delta,N}(\cdot,z)\|_{L^2(\C)}\\
&\leq C(p)\,\left(\frac{R}{\delta}\right)^{\frac{2}{p}-1} \le
|k|^{(\frac{2b}{\epsilon}) (\frac{2}{p}-1)}
\endaligned$$
Here we have inserted the values of $R$ and $\delta$ from \textbf{II} and \textbf{III}.
Thus, whenever $|k|\geq 2R\ge \frac{|k|}{2}$ we have
\begin{equation*}\label{final}
\aligned
\left|\int_{|\xi|<R}\widehat{g}(\xi,k)\,\widehat{\Phi_{\delta,N}}(\xi,z)\,dA(w)\right|
&\leq
C(\alpha,\kappa,p)\,\frac{\Gamma_0}{|k|^\alpha}\,M(p)^{n_0}\,
|k|^{(\frac{2b}{\epsilon}) (\frac{2}{p}-1)}
\endaligned
\end{equation*}
Now since $\|T\|_{L^p} \le C(p-1)$  it follows that the best
choice is $p=4/3$. Inserting this and the value of $n_0$ from
(\ref{no}) in the previous equation,
\begin{equation*}
C(\alpha,\kappa,p)\,\Gamma_0\,\frac{1}{|k|^{\alpha}}\,|k|^{C(\kappa)\,b}\,
|k|^{\frac{b}{\epsilon}} \le C(\alpha,\kappa,p)\,\Gamma_0 |k|^{ b
C(\kappa)\epsilon^{-1} -\alpha}
\end{equation*}
Finally we want that ($\mathbf{IV}$) is controlled by $k^{-b}$ as well. Since $\epsilon=\epsilon(\kappa)<1$ and we already asked $b<\epsilon$, we end up getting that it suffices that
$$
b<\min\left\{\frac{\epsilon\alpha}{C },\epsilon\right\}=\frac
{\epsilon \alpha}{C}.
$$
Here $C=C(\kappa)>1$ and we have use that $\alpha<1$. The proof is concluded.
\end{proof}

\subsection{Estimates for the nonlinear equation}\label{nonlinearsubsection}

Now that the behavior at $k\to\infty$ of the solutions to the linearized equation \eqref{almostnonlinearbeltrami} is known, it is time to study the behavior of the complex geometric optics
solutions.

\begin{Theorem}\label{nonlineardecay}
Let $\mu\in W^{\alpha, 2}(\C)$ be a Beltrami coefficient, real valued, compactly supported in $\frac{1}{4}\D$, such that $\|\mu\|_\infty\leq\frac{K-1}{K+1}$ and $\|\mu\|_{W^{\alpha, 2}(\C)}\leq\Gamma_0$. Let $\phi=\phi_\lambda(z,k)$ be the solution to
$$
\begin{cases}
\overline\partial\phi_{\lambda}(z, k)=-\displaystyle\frac{\overline k}{k}\,\lambda\,\mu(z)\,e_{-k}(\phi_{\lambda}(z, k))\,\overline{\partial\phi_{\lambda}(z, k)}\\
\phi_\lambda(z,k)-z={\cal O}(1/|z|)\text{ as }|z|\to\infty.
\end{cases}
$$
There exists constants $C=C(K)>0$ and $b=b(K)$ such that
$$|\phi_\lambda(z,k)-z|\leq \frac{C\,\Gamma_0^\frac1K}{|k|^{b\alpha}}$$
for every $z\in\C$, $k\in\C$ and $\lambda\in\partial\D$.
\end{Theorem}
\begin{proof}
Since the estimate we look for is uniform in $z$ and $\lambda$, it suffices to show equivalent decay for the inverse mapping $\psi_\lambda=\phi_\lambda^{-1}$. But $\psi_\lambda$ is the only quasiconformal mapping on the plane that satisfies both the equation
$$
\overline\partial\psi_\lambda(z,k)=\frac{\overline{k}}{k}\,\lambda\,e_{-k}(z)\,\mu(\psi_\lambda(z,k))\,{\partial\psi_\lambda(z,k)}
$$
(compare with \eqref{almostnonlinearbeltrami}) and the condition
$\psi_\lambda(z,k)-z={\cal O}(1/|z|)$ as $|z|\to\infty$. Then, we just need to show that the coefficient
$$\mu(\psi_\lambda(z,k))$$
satisfies the assumptions of Proposition \ref{lineardecay}. First, it is obvious that
$$\|\mu\circ\psi_\lambda(\cdot, k)\|_{\infty}\leq\frac{K-1}{K+1}$$
and it is also obvious that $\supp(\mu\circ\psi_\lambda(\cdot,k))\subset\D$ (this follows from Koebe's $\frac14$ Theorem). Then, it remains to prove that $\mu\circ\psi_\lambda\in W^{\beta,2}(\C)$ for some $\beta\in(0,1)$. But this follows from Proposition \ref{composition}. Indeed, since $\mu\in W^{\alpha, 2}(\C)\cap L^\infty(\C)$, we have $\mu\circ\psi_\lambda\in W^{\beta,2}(\C)$ with
$$
\|\mu\circ\psi_\lambda\|_{W^{\beta,2}(\C)}\leq C\,\Gamma_0^\frac1K
$$
for any $0<\beta<\frac{\alpha}{K}$, where $C=C(\alpha,\beta,K)$. Note also that $\beta$ behaves linearly as a function of $\alpha$, with constant depending only on $K$. So the result follows.
\end{proof}

\begin{Rem}\label{sizeofsupport}
In the above result, the assumption $\supp(\mu)\in\frac14\D$ is not restrictive. Indeed,
if $\supp(\mu)\subset D(0,R)$ for some $R>0$ then the function $\mu_R(z)=\mu(4Rz)$ defines a new Beltrami coefficient, compactly supported in $\frac{1}{4}\D$, does not change the ellipticity bound, and
$$\aligned
\|D^\alpha\mu_R\|_{L^2(\C)}&=(4R)^{1-\alpha}\,\|D^\alpha\mu\|_{L^2(\C)}.
\endaligned$$
One can then apply the previous Theorem to this coefficient $\mu_R$ and obtain estimates for the complex geometric optics solutions. But
$f_{\mu_R}(z,k)=f_\mu(4Rz, \frac{k}{4R})$
and in fact if we represent these solutions as $f_{\mu}(z,k)=\exp(ik\phi_\mu(z,k))$, then
$$
\phi_{\mu_R}(z,k)=\frac{1}{4R}\,\phi_\mu\left(4Rz, \frac{k}{4R}\right).
$$
so the estimates for $\phi_{\mu_R}$ coming from the previous theorem give similar estimates for $\phi_\mu$, modulo a power of $R$.
\end{Rem}
\noindent Now as discovered in \cite{AP} the unimodular complex parameter $\lambda$ allows to push the decay estimates to complex geometric optics solutions to the $\gamma$-harmonic equation. As always, given a
real Beltrami coefficient $\nu$ we denote by $f_\nu(z,k)=e^{ikz}\,M_\nu(z,k)$ the complex geometric optics
solutions to $\overline\partial f=\nu\,\overline{\partial f}$.

\begin{Theorem}\label{2ndorderdecay}
Let $\mu$ be as in Theorem \ref{nonlineardecay}, and define
$$u=\Re(f_\mu)+i\,\Im(f_{-\mu}).$$
There exist a function $\epsilon=\epsilon(z,k)$ and positive constants $C=C(K)$ and $b=b(K)$ such that
\begin{enumerate}
\item[(a)] $u(z,k)=e^{ik(z+\epsilon(z,k))}$.
\item[(b)] $|\epsilon(z,k)|\leq\displaystyle\frac{C \, \Gamma_0^\frac{1}{K}}{|k|^{b \, \alpha}}$ for each $z,k\in\C$.
\end{enumerate}
Further, a similar estimate holds for $\tilde{u}=\Re(f_{-\mu})+i\,\Im(f_{\mu}).$
\end{Theorem}
\begin{proof}
A calculation shows that $u$ may be rewritten as
$$
u=f_\mu\, \frac{1+ \displaystyle\frac
{\overline{f_\mu}-\overline{f_{-\mu}}} {f_\mu+f_{-\mu}} } {1+
\displaystyle\frac {f_\mu-f_{-\mu}} {f_\mu+f_{-\mu}}}.
$$
Thus, the Theorem will follow if we find a function
$\epsilon(z,k)$ such that $|\epsilon(z,k)|\leq\frac{C \,
\Gamma_0^\frac1K}{|k|^{b \, \alpha}}$ and
$$
\left| \frac{f_\mu-f_{-\mu}}{f_\mu+f_{-\mu}} \right|\leq
1-e^{|k\,\epsilon(z,k)|}
$$
Following \cite[Lemma 8.2]{AP}, it suffices to see that
$$
\inf_t\,\left|\frac{f_\mu-f_{-\mu}}{f_\mu+f_{-\mu}}+e^{it}\right|\geq
e^{|k\,\epsilon(z,k)|}
$$
For this, define
$\Phi_t(z,k)=e^{-\frac{it}{2}}\left(f_\mu\,\cos(t/2)+i\,f_{-\mu}\,\sin(t/2)\right)$.
It follows easily that for each fixed $k$,
$$
\begin{cases}
|e^{-ikz}\,\Phi_t(z,k)-1|={\cal O}(1/z)&\text{ as }|z|\to\infty\\
\overline\partial\Phi_t=e^{-it}\mu\,\overline{\partial\Phi_t}
\end{cases}
$$
Thus, by uniqueness in Theorem \ref{existuniqbeltrami}, $\Phi_t$
is nothing but the complex geometric optics solution
$\Phi_t=f_{\lambda\mu}$ with $\lambda=e^{-it}$. But then
$$
\frac{f_\mu-f_{-\mu}}{f_\mu+f_{-\mu}}+e^{it}=\frac{2\,e^{it}\,\Phi_t}{f_\mu+f_{-\mu}}=\frac{f_{\lambda\mu}}{f_\mu}\,\frac{2\,e^{it}}{1+\frac{M_{-\mu}}{M_\mu}}
$$
On the other hand, from Theorem \ref{nonlineardecay} we get that
$$e^{-|k\,\epsilon(z,k)|}\leq|M_\mu(z,k)|=|e^{ik(\phi_\mu(z,k)-z)}|\leq e^{|k\,\epsilon(z,k)|}$$
where $|\epsilon(z,k)|\leq\frac{C\, \Gamma_0^\frac1K}{|k|^{b \, \alpha}}$
and
$$e^{-2|k\,\epsilon(z,k)|}\leq\frac{|f_{\lambda\mu}(z,k)|}{|f_\mu(z,k)|}\leq e^{2|k\,\epsilon(z,k)|}$$
uniformly for $\lambda\in\partial\D$. Finally, by Theorem
\ref{existuniqbeltrami}, we also have
$\Re\left(\displaystyle\frac{M_{-\mu}}{M_\mu}\right)>0$, so that
the result follows.
\end{proof}

\section{Proof of Theorem \ref{Thetheorem}}\label{endofproof}

In order to get stability from $\Lambda_\gamma$ to $\mu$, we will need the stability result for the complex geometric optics solutions in terms of a Sobolev norm. It comes as an interpolating consequence of the $L^\infty$ stability result given at Theorem \ref{2ndorderdecay} and of the regularity of the solutions to a
Beltrami equation with Sobolev coefficients (see Theorem \ref{regcgos}).

\begin{Theorem}\label{sobolevstabilitycgos}
Let $\mu_1,\mu_2$ be Beltrami coefficients, compactly supported in $\D$, such that  $\|\mu_j\|\leq\frac{K-1}{K+1}$ and $\|\mu_j\|_{W^{\alpha,2}(\C)}\leq\Gamma_0$. Let $f_{\mu_j}$ denote
the complex geometric optics solutions to $\overline\partial f_{\mu_j}=\mu_j\,\overline{\partial f_{\mu_j}}$. Then, for each $\theta\in(0,\frac{1}{K})$ we have
$$
\|f_{\mu_1}-f_{\mu_2}\|_{\dot{W}^{1,(1+\alpha\theta)2}(\D)}\leq
C\,\left(1+\Gamma_0\right)^\frac{1}{K}\,\left|\log\frac1\rho\right|^{-b\alpha^2}
$$
for come constants $C=C(|k|,K)>0$ and $b=b(K)>0$. In particular, the same bound holds with the $\dot{W}^{1,2}(\D)$-norm.
\end{Theorem}
\begin{proof}
The subexponential growth obtained in Theorem \ref{2ndorderdecay} entitled us to apply Theorem \ref{teoremaBFR} {\textbf{B}} to the solutions $u_{\gamma_i}$. Since they are equivalent to the corresponding $f_\mu$ we achieve the estimate
$$
\|f_{\mu_1}(\cdot,k)-f_{\mu_2}(\cdot,k)\|_{L^\infty(\D)}\leq \frac{C\,\Gamma_0^\frac{1}{K^2}}{|\log(\rho)|^{b\alpha}}$$
for some positive constants $C=C(k,K)$ and $b=b(K)$. On the other hand, from Theorem \ref{regcgos}, for every
$\theta\in(0,\frac{1}{K})$ we have
$$\aligned
\|f_{\mu_1}(\cdot,k)-f_{\mu_2}(\cdot,k)\|_{\dot{W}^{1+\alpha\theta,2}(\D)}&=\|D^{1+\theta\alpha}\left(f_{\mu_1}(\cdot,k)-f_{\mu_2}(\cdot,k)\right)\|_{L^{2}(\D)}\\
&\leq C\,e^{C|k|}\,\left(1+|k|\right)\,\left(\Gamma_0+|k|^\alpha\right)^\theta.
\endaligned$$
As in Theorem \ref{regcgos}, here $C$ may depend on $K$. Let $\varphi\in{\cal C}^\infty(\C)$ be a cut-off function, compactly supported in $2\D$, such that $\varphi|_\D=\chi_{|\D}$. Then the above estimates imply that
$$\aligned
&\|(f_{\mu_1}(\cdot,k)-f_{\mu_2}(\cdot,k))\varphi\|_{L^\infty(\C)}\leq\frac{C\,\Gamma_0^\frac{1}{K^2}}{|\log(\rho)|^{b\alpha}}\\
&\|(f_{\mu_1}(\cdot,k)-f_{\mu_2}(\cdot,k))\,\varphi\|_{\dot{W}^{1+\alpha\theta,
2}(\C)}\leq C\,e^{C|k|}\,\left(1+|k|\right)\,\left(\Gamma_0+|k|^\alpha\right)^\theta
\endaligned$$
where, as usually, $\|\cdot\|_{\dot{W}^{s,p}(\C)}$ denotes the homogeneous Sobolev norm.
Now, an interpolation argument shows that for each $0<\beta<1$ we have
$$\aligned
\|\varphi(f_{\mu_1}(\cdot,k)-&f_{\mu_2}(\cdot,k))\|_{\dot{W}^{(1+\alpha\theta)\beta,\frac{2}{\beta}}(\C)}\\
&\leq\,C\,\,e^{C\beta|k|}\,\left(1+|k|\right)^\beta\,\left(\Gamma_0+|k|^\alpha\right)^{\beta\theta}\,\frac{\Gamma_0^\frac{1-\beta}{K^2}}{|\log(\rho)|^{b\alpha(1-\beta)}}\\
&\leq\,C\,\,e^{C\beta|k|}\,\left(1+|k|\right)^{\beta(1+\alpha\theta)}\,\left(1+\Gamma_0\right)^{\frac{1}{K^2}+\beta(\theta-\frac{1}{K^2})}\,\frac{1}{|\log(\rho)|^{b\alpha(1-\beta)}}
\endaligned$$
where $C=C(K)$. In particular, for $\beta=\frac{1}{1+\alpha\theta}$ we get that
$$\aligned
\|f_{\mu_1}-&f_{\mu_2}\|_{\dot{W}^{1,(1+\alpha\theta)2}(\D)}
\leq\|(f_{\mu_1}-f_{\mu_2})\,\varphi\|_{W^{1,(1+\alpha\theta)2}(\C)}\\
&\leq
C\,\,e^{C|k|}\,\left(1+|k|\right)\,\left(1+\Gamma_0\right)^{\frac{1}{K^2}+\frac{1}{1+\alpha\theta}(\theta-\frac{1}{K^2})}\,\frac{1}{|\log(\rho)|^{\frac{b\alpha^2\theta}{1+\alpha\theta}}}
\endaligned$$
Here the sharp modulus of continuity is obtained when the logarithm has the bigger exponent, which is given for $\theta=\frac1K$. Thus, we end up with
$$
\|f_{\mu_1}-f_{\mu_2}\|_{\dot{W}^{1,(1+\alpha\theta)2}(\D)}\leq
C(|k|)\,\left(1+\Gamma_0\right)^\frac{1}{K}\,\frac{1}{|\log(\rho)|^{\frac{b}{K}\,\alpha^2}}
$$
as claimed.
\end{proof}
\noindent
It just remains to see how the previous estimate drives us to the
final stability bounds for the Beltrami coefficients (and therefore for the conductivities). To do
this, the following interpolation Lemma will be needed. Note that it includes $L^p$ spaces with $p<1$.


\begin{Lemma}[Interpolation] \label{interpolation}
Let $0<p_0\leq 2$ and $2<p_1 \leq\infty$. Let $\theta$ be such that
\[ \frac{1}{2}=\frac{\theta}{p_0}+\frac{1-\theta}{p_1}. \]
Then
\[ \|f\|_{L^2} \le \|f\|_{L^{p_0}}^\theta \|f\|_{L^{p_1}}
^{1-\theta} \] for any $f\in L^{p_0}\cap L^{p_1}$.
\end{Lemma}
\begin{proof}
The proof is adapted for the usual Riesz method for interpolation with a little extra care when $p_0<1$. We choose $r<p_0$ and define exponents $q_0,q_1,q_2$ such that
$$
\frac{1}{r}=\frac{1}{p_0}+\frac{1}{q_0}\hspace{1.5cm}
\frac{1}{r}=\frac{1}{2}+\frac{1}{q_2}\hspace{1.5cm}
\frac{1}{r}=\frac{1}{p_1}+\frac{1}{q_1}
$$
For $z=x+iy$ in the strip $0 \le y \le 1$, we define the analytic function
\[G(z)= |g|^{q_2(\frac{z}{q_0}+\frac{1-z}{q_1})}\frac{g}{|g|}. \]
Notice that $|G(iy)|^{q_1}=|g|^{q_2}$, $|G(1+iy)|^{q_0}=|g|^{q_2}$, and $|G(\theta+iy)|=|g|$. Now we introduce the function
\[ I(z)= \left(\int |f|^r\,|G(z)|^r\right)^{\frac{1}{r}} \]
We can estimate its values at the boundary of the strip,
$$\aligned
|I(iy)| &\le \|f\|_{L^{p_1}}\,\|G(iy)\|_{L^{q_1}}=\|f\|_{L^{p_1}} \left(\int |g|^{q_2}\right)^{\frac{1}{q_1}}, \\
|I(1+iy)| &\le
\|f\|_{L^{p_0}}\,\|G(iy)\|_{L^{q_0}}=\|f\|_{L^{p_0}} \left(\int|g|^{q_2}\right)^{\frac{1}{q_0}}.
\endaligned
$$
Then we apply Phragmen-Lindel\"of theorem to the function $I(z)$ obtaining that
$$
\aligned I(\theta+iy) &\le \left(\|f\|_{L^{p_1}} \left(\int
|g|^{q_2}\right)^{\frac{1}{q_1}}\right)^{1-\theta}
\left(\|f\|_{L^{p_0}} \left(\int |g|^{q_2}\right)^{\frac{1}{q_0}}\right)^{\theta} \\
&\le \|g\|_{L^{q_2}}
\|f\|_{L^{p_0}}^{\theta}\,\|f\|_{L^{p_1}}^{1-\theta}
\endaligned
$$
But $I(\theta+iy)= \|fg\|_{L^r}$, so the result follows.
\end{proof}
\noindent We are finally led to obtain the desired stability in $L^2$ norm of the Beltrami coefficients.

\begin{Cor}[Proof of Theorem~\ref{Thetheorem}]\label{finalcoroll}
Let $\mu_1,\mu_2$ be Beltrami coefficients, compactly supported in
$\D$, such that $\|\mu_j\|\leq\frac{K-1}{K+1}$ and
$\|\mu_j\|_{W^{\alpha,2}(\C)}\leq\Gamma_0$. There exists constants
$b=b(K)>0$ and $C=C(\alpha,K)>0$ such that
$$
\|\mu_1-\mu_2\|_{L^2(\D)}\leq
C\,\left(1+\Gamma_0\right)^\frac{1}{K^2}\,\left|\log\frac1\rho\right|^{-b\alpha^2}$$
where $\rho=\|\Lambda_1-\Lambda_2\|_{H^\frac{1}{2}(\partial(\D)\to
H^{-\frac{1}{2}}(\partial\D)}$.
\end{Cor}
\begin{proof}
Denote by $f_i$ the complex geometric optics solution $f_{\mu_i}$
of $\overline\partial f=\mu_i\,\overline{\partial f}$ with $k=1$.
Then,
$$\aligned
|\mu_1-\mu_2|
&=\left|\frac{\overline\partial f_1\,\overline{\partial f_2}-\overline\partial f_2\,\overline{\partial f_1}} {\overline{\partial f_1}\,\overline{\partial f_2}}\right|
=\left|\frac{-\overline\partial f_1\,(\overline{\partial f_1}-\overline{\partial f_2})+(\overline\partial f_1-\overline\partial f_2)\overline{\partial f_1}} {\overline{\partial f_1}\,\overline{\partial f_2}}\right|\\
&\leq\frac{|\overline\partial f_1-\overline\partial f_2|}{|\partial f_2|}+|\mu_1|\,\frac{|\partial f_1-\partial f_2|}{|\partial f_2|}\leq 2\,\frac{|Df_1-Df_2|}{|\partial f_2|}
\endaligned
$$
because $|Df_j|=|\partial f_j|+|\overline\partial f_j|$. Therefore, for any $s>0$
$$
\|\mu_1-\mu_2\|_{L^s(\D)}\leq
2\left\|\frac{Df_1-Df_2}{\overline{\partial
f_2}}\right\|_{L^s(\D)}.
$$
Now, let $p\in(0,\frac{2}{K-1})$ and $\theta\in(0,1/K)$. Then put
$\frac{1}{s}=\frac{1}{2(1+\alpha \theta}+\frac{1}{p}$. An application of H\"older's inequality gives us that
$$\aligned
\left\|\frac{D f_1-D f_2}{\overline{\partial
f_2}}\right\|_{L^s(\D)}
\leq C(\alpha,\theta)\,\|D f_1-D f_2\|_{L^{2(1+\alpha\theta)}(\D)}\,\left\|\frac{1}{\partial f_2}\right\|_{L^p(\D)}\\
\leq C(\alpha,\theta)\,\|f_1-f_2\|_{\dot{W}^{1,2(1+\alpha\theta)}(\D)}\,\left\|\frac{1}{\partial
f_2}\right\|_{L^p(\D)}.
\endaligned$$
Now, using Lemma \ref{inversecgos} and Theorem \ref{sobolevstabilitycgos}, we obtain the estimate
$$
\|\mu_1-\mu_2\|_{L^s(\D)}\leq
C\,\left(1+\Gamma_0\right)^\frac{1}{K}\,\left|\log\frac1\rho\right|^{-b\alpha^2}
$$
where $C>0$ depends on $\alpha, \theta$,  and $K$, and $b>0$
depends on $K$. Finally, if $s\geq 2$ then we are done, since
$\mu_i$ are compactly supported in $\D$. But in general we only
know $0<\frac{2}{K}<s$ so one could well have $s<1$. In this case,
  in order to get $L^2$ estimates only interpolation between
$L^s$ and $L^\infty$ is needed, as in Lemma \ref{interpolation}.
Namely,
$$
\|\mu_1-\mu_2\|_{L^2(\D)}\leq \|\mu_1-\mu_2\|_{L^s(\D)}^\frac{s}{2}\,\|\mu_1-\mu_2\|_{L^\infty(\D)}^\frac{2-s}{2}
$$
and now the stability estimate looks like
$$
\|\mu_1-\mu_2\|_{L^2(\D)}\leq
C\,\left(1+\Gamma_0\right)^\frac{1}{K^2}\,\left|\log\frac1\rho\right|^{-b\alpha^2}
$$
where the constants may have changed.
\end{proof}

\vskip 1cm
\begin{itemize}

\item[]{A. Clop\\
www.mat.uab.cat/$\sim$albertcp\\
Department of Mathematics and Statistics,\\ 
P.O.Box 35 (MaD)\\
FI-40014 University of Jyv\"askyl\"a\\
Finland}

\item[]{D. Faraco, A. Ruiz\\
Departmento de Matem\'aticas\\
Universidad Aut\'onoma de Madrid\\
Campus de Cantoblanco, s/n\\
28049-Madrid\\
Spain}

\end{itemize}

\end{document}